\documentclass[reqno]{amsart}
\usepackage{amsmath}

\usepackage{mathrsfs}

\usepackage{mathpazo}

\usepackage{url}
\usepackage{bbm}

\makeatletter \@addtoreset{equation}{section} \makeatother

\renewcommand\thefigure{\thesection.\@arabic\c@figure}
\renewcommand\thetable{\thesection.\@arabic\c@table}
\newtheorem{theorem}{Theorem}[section]
\newtheorem{lemma}[theorem]{Lemma}
\newtheorem{proposition}[theorem]{Proposition}
\newtheorem{corollary}[theorem]{Corollary}

\theoremstyle{remark}

\newtheorem{remark}[theorem]{Remark}

\newcommand{\mc}[1]{{\mathscr #1}}

\newcommand{\bb}[1]{{\mathbb #1}}

\newcommand{\<}{\langle}
\renewcommand{\>}{\rangle}

\DeclareMathOperator{\diag}{diag}

\newcommand{\sfrac}[2]{{\smash{\frac{#1}{#2}}}}

\usepackage{zito}

\begin{document}

\title{Quadratic fluctuations of the simple exclusion process}

\author{Milton Jara}

\begin{abstract}
We introduce a two-dimensional, distribution-valued field which we call the {\em quadratic field} associated to the one-dimensional Ornstein-Uhlenbeck process. 
We show that the stationary quadratic fluctuations of the simple exclusion process, when rescaled in the diffusive scaling, converge to this quadratic field. We show that this quadratic field evaluated at the diagonal corresponds to the Wick-renormalized square of the Ornstein-Uhlenbeck process, and we use this new representation in order to prove some small and large-time properties of it.
\end{abstract}

\address{\noindent IMPA, Estrada Dona Castorina 110, CEP 22460 Rio de
  Janeiro, Brasil
\newline
e-mail: \rm \texttt{mjara@impa.br}}

\maketitle

\section{Introduction}

In recent years, scaling limits of nonlinear and/or singular functionals of stochastic lattice models have attracted a lot of attention. Just to give a couple of examples, we mention the extensive studies of the KPZ universality class (see \cite{Cor} for a review) and Gaussian multiplicative chaos associated to Liouiville quantum gravity (see \cite{RhoVar} for a review). In \cite{Hai2} the author has proposed a general framework (the so-called theory of  {\em regularity structures}) in order to deal with ill-posed stochastic PDE's on which the trouble comes from a nonlinear term (like in the KPZ or stochastic Allen-Cahn equations) or from a singular linear term (like in the parabolic Anderson model). The theory of regularity structures allows to make sense of troublesome equations in a meaningful way. Moreover, various scaling limits of stochastic lattice models on which these singular and/or nonlinear observables play an important role should be given in terms of solutions to these equations. However, aside from models on which a great deal of integrability is present (the term {\em stochastic integrability} was coined in \cite{Spo}), the question of convergence of nonlinear fluctuations of stochastic lattice models is basically open; see however \cite{GJ}, \cite{GJS}.

One of the main ingredients of the theory of regularity structures consists in making sense {\em a priori} of enough nonlinear and/or singular observables of solutions of {\em linear} stochastic PDE's. In the case of the KPZ equation \cite{Hai1}, one starts with the solution of the Ornstein-Uhlenbeck equation
\begin{equation}
d \mc Y_t = \Delta \mc Y_t + \tfrac{1}{\sqrt 2} \nabla d \omega_t
\end{equation}
where $d \omega_t$ is a space-time white noise
and tries to make sense of various nonlinear functionals of it (the twelve tree-labeled processes in \cite{Hai1}). To avoid uncomfortable issues arising from the lack of compactness, \cite{Hai1} restricts himself to the circle $\bb T$. The simplest of these processes corresponds to $\mc Y_t(x)^2$. Since $\mc Y_t$ turns out to be a distribution, it is far from clear how to define $\mc Y_t(x)^2$. Let us restrict ourselves to the stationary situation, on which for any fixed time $t$, $\mc Y_t$ is a spatial white noise of variance $\chi =\frac{1}{4}$. The simplest choice should be to take an approximation of the identity $\iota_\varepsilon(x)$ centered at $x \in \bb T$, and to consider $\mc Y_t(x)^2$ as the limit of $\mc Y_t(\iota_\varepsilon(x))^2$ as $\epsilon \to 0$ in some sense. It turns out that this plan can be formalized after a Wick renormalization: when integrated in time and space against a smooth test function, $\mc Y_t(\iota_\varepsilon(x))^2 - \frac{\chi}{\varepsilon}$ has a meaningful limit as $\varepsilon \to 0$. Of course, if one wants to use \cite{Hai1} in order to obtain a unique solution of the KPZ equation, there are still 11 processes to go, but on this paper we just focus on this one process (we don {\em not} claim we can treat the other 11 processes!). Up to our knowledge, this {\em squared field} was first considered in \cite{Ass}, where the convergence of space-time fluctuations of two-point functions to the squared field of the Ornstein-Uhlenbeck process was obtained. A more general, different proof was implicitly obtained in \cite{GJ} and explicitly stated in \cite{GJ3}, as a part of a program towards the derivation of the KPZ equation from general stochastic lattice models.

In this paper we propose a new, different approach in order to define the field $\mc Y_t(x)^2$. The {\em quadratic field} associated to $\mc Y_t$ is the two-dimensional process formally defined as $Q_t(x,y) = \mc Y_t(x) \mc Y_t(y)$. In order to define this object in a weak sense, some care needs to be taken at the diagonal; the simplest choice is to take $Q_t(x,x)=0$ (anyway $Q_t$ is just a formal object). Blindly applying It\^o's formula we see that $Q_t$ satisfies an equation of the form
\begin{equation}
\label{ecintro.1}
d Q_t = \Delta Q_t + d \mc W_t,
\end{equation} 
where $\mc W_t$ is a distribution-valued martingale that can be computed in terms of the Ornstein-Uhlenbeck process $\mc Y_t$ and the noise $d\omega_t$. It turns out that this equation is well-posed and solutions of it can be constructed straightforwardly. Let $\iota_\varepsilon(x,y)$ be a two-dimensional approximation of the identity. Let $f: \bb T \to \bb R$ be a mean-zero, regular function and consider the process
\begin{equation}
\mc A_t^\varepsilon(f) = \int_0^t Q_s\big( f(x) \iota_\varepsilon(x-y,x-y)\big) ds.
\end{equation}
It turns out that the process $\mc A_t^\varepsilon$ has a non-trivial limit as $\varepsilon \to 0$ and this limit coincides with the squared process $\mc Y_t(x)^2$ constructed in \cite{Ass}. The main difference is that now we obtain the squared process as a singular {\em linear} observable of the solution of a reasonably well-behaved stochastic PDE. The main advantage of this representation is that solving a simple Laplace problem, we can obtain various properties of the squared process in a more or less straightforward way.

In order keep the paper at a reasonable length, we focus on two issues about the quadratic field $Q_t$ and we keep the model as simple as possible. First we show that the quadratic fluctuations of the simple exclusion process converge to the quadratic field $Q_t$. It is well known that the simple exclusion process is integrable in the sense that various quantities of interest, among them $n$-point correlation functions, can be computed almost explicitly. We do {\em not} take full advantage of this feature. Given the technical Boltzmann-Gibbs principle for granted, a careful reading of our proof shows that this convergence result can be extended for the speed-change exclusion processes considered in \cite{GJ}. This technical principle has been proved in \cite{D-MIPP}. And then we obtain short-time and long-time properties of the process $\mc A_t$, using the construction outlined above. These properties have not been obtained before, and they are good examples of the advantages of our construction with respect to previous ones.

This paper is organized as follows. In Section \ref{s1} we define the exclusion process and we define the various fluctuation fields we want to study on this article. On the way, we provide various topological definitions which are needed to handle distribution-valued processes. We also state the main results of the article. 
%%%
%%%
In Section \ref{s2} we show that the discrete quadratic fields form a relatively compact sequence of distributions with respect to the $J_1$-Skorohod topology on the space of distribution-valued, {\em c\`adl\`ag} paths.
%%%
%%%
In Section \ref{s3} we show that the discrete quadratic fields converge to the unique stationary solution of equation \eqref{ecintro.1}. The proof is divided in three parts. Recall that limit points exists due to tightness. In Section \ref{s3.1} we show that any limit point is continuous; this will be important later when characterizing some martingales in terms of quadratic variations. In Section \ref{s3.2} we show that various martingales associated to the discrete quadratic fields have limit points which are martingales. This shows that any limit point of the discrete quadratic fields satisfies a martingale formulation of \eqref{ecintro.1}.
In Section \ref{s3.3} we show that this martingale problem has a unique solution, which closes the proof of the convergence result.
%%%
%%%
In Section \ref{s4}

\section{The model}
\label{s1}
\subsection{The simple exclusion process}
Let $\bb T_n = \smash{\frac{1}{n}}\bb Z / \bb Z$ denote the discrete circle with $n$ points. Let $\Omega_n= \{0,1\}^{\bb T_n}$ be the state space of a continuous-time Markov chain which we describe as follows. We denote by $\eta = \{\eta(x); x \in \bb T_n\}$ the elements of $\Omega_n$ and we call them {\em configurations}. We consider the set $\bb T_n$ as embedded in the continuous circle $\bb T = \bb R / \bb Z$. We call the elements of $\bb T_n$ {\em sites}. We say that two sites $x,y \in \bb T_n$ are {\em neighbors} if $|y-x|=\smash{\frac{1}{n}}$. In this case we write $x \sim y$.
To each pair $\{x,y\}$ of neighbors we attach a Poisson clock of rate $n^2$. Each Poisson clock is independent of the other clocks. Each time a Poisson clock rings, we exchange the occupation numbers of the corresponding pair of neighbors. For each $\eta \in \Omega_n$ and each $x,y \in \bb T_n$, let $\eta^{x,y} \in \Omega_n$ denote the configuration obtained from $\eta$ by exchanging the occupation numbers at $x$ and $y$, that is,
\begin{equation}
\eta^{x,y}(z)=
\begin{cases}
\eta(y), & z=x\\
\eta(x), &z=x\\
\eta(z), &z \neq x,y.
\end{cases}
\end{equation}
This informal description corresponds to the Markov chain $\{\eta_t^n; t \geq 0\}$ generated by the operator $L_n$, given by
\begin{equation}
L_n f(\eta) = n^2\sum_{x \sim y} \nabla_{x,y} f(\eta)
\end{equation}
for any $f: \Omega_n \to \bb R$, where $\nabla_{x,y} f: \Omega_n \to \bb R$ is given by $\nabla_{x,y} f(\eta) = f(\eta^{x,y})-f(\eta)$ for any $x,y \in \bb T_n$ and any $\eta \in \Omega_n$. The sum is over {\em unordered pairs} $\{x,y\}$ of neighbors in $\bb T_n$.

We say that a site $x$ is {\em occupied} by a particle at time $t \geq 0$ if $\eta_t^n(x)=1$. If $\eta_t^n(x)=0$, we say that the site $x$ is {\em empty} at time $t \geq 0$. With this convention about particles and empty sites (or {\em holes}), the dynamics of $\{\eta_t^n; t \geq 0\}$ has the following interpretation. Each particle tries to jump to each of its two neighbours with exponential rate $n^2$. At each attempt, it verifies whether the destination site is empty, on which case it jumps to it. Otherwise the particle stays where it is. This particle interpretation gives the name {\em simple exclusion process} to the family of  processes $\{\eta_t^n; t \geq 0\}$. Notice that particles are neither created nor annihilated by this dynamics. By reversibility, it is easy to check that the uniform measures on the spaces
\begin{equation}
\Omega_{n,\ell} = \Big\{\eta \in \Omega_n; \sum_{x \in \bb T_n} \eta(x) = \ell\Big\}
\end{equation}
are invariant with respect to the dynamics of $\{\eta_t^n;t \geq 0\}$ for any $\ell \in \{0,1,\dots,n\}$. Let us call $\nu_{n,\ell}$ these measures. Checking the irreducibility of the sets $\Omega_{n,\ell}$ with respect to the dynamics, it can be concluded that $\nu_{n,\ell}$ is actually ergodic under the evolution of $\{\eta_t^n; t \geq 0\}$ for any $\ell$. Notice that the product measures $\nu_\rho$ given by \begin{equation}
\nu_\rho(\eta) = \prod_{x \in \bb T_n} \big\{ \rho \eta(x) + (1-\rho) (1-\eta(x))\big\}
\end{equation}
are invariant and reversible under $\{\eta_t^n; t \geq 0\}$ for any $\rho \in [0,1]$. However, these measures are not ergodic due to the conservation of the number of particles. In fact, these measures are obtained as proper convex combinations of the measures $\{\nu_{n,\ell}; \ell=0,1,\dots,n\}$. From now on we start the process $\{\eta_t^n; t \geq 0\}$ from the invariant measure $\nu_{\rho}$ for $\rho = \smash{\frac{1}{2}}$. To avoid uninteresting topological issues, we fix $T>0$ and we restrict the process $\{\eta_t^n;t \geq 0\}$ to the interval $[0,T]$. We denote by $\bb P_n$ the distribution of $\{\eta_t^n; t \in [0,T]\}$ on the space $\mc D([0,T]; \Omega_n)$ of c\`adl\`ag trajectories from $[0,T]$ to $\Omega_n$, and we denote by $\bb E_n$ the expectation with respect to $\bb P_n$. The expectation with respect to $\nu_{n,\ell}$ will be denoted by $E_{n,\ell}$, and the expectation with respect to  $\nu_\rho$ will be denoted by $E_\rho$.

\subsection{The fluctuation fields} 
In order to define the fluctuation fields associated to the process $\{\eta_t^n; t \in [0,T]\}$ in a proper way, we need to introduce some notations and some topologies.
For a given compact set $\mc K$ with a differentiable structure, let $\mc C^\infty(\mc K)$ denote the set of infinitely differentiable functions $f: \mc K \to \bb R$. We will only consider $\mc K = \bb T^{\ell}$ for $\ell = 1,2$, that is, the circle $\bb T$ and the two-dimensional torus $\bb T^2$. The space $\mc C^\infty(\mc K)$ is a Polish space with respect to the topology generated by the metric $d: \mc C^\infty(\mc K) \times \mc C^\infty(\mc K) \to [0,\infty)$ given by
\begin{equation}
d(f,g) = \sum_{\ell \in \bb N_0} \frac{1}{2^\ell} \min\{1, \|f^{(\ell)}-g^{(\ell)}\|_\infty\}.
\end{equation}
Here $f^{(\ell)}$ denotes the $\ell$-derivative of $f$, which is a function in the case $\mc K = \bb T$ and an $\ell$-dimensional, symmetric tensor in the case $\mc K = \bb T^2$.
The space $\mc D(\mc K)$ (do not confuse with $\mc D([0,T]; \Omega_n)$) will denote the set of linear, continuous functions $\varphi: \mc C^\infty(\mc K) \to \bb R$. In other words, $\mc D(\mc K)$ is the topological dual of $\mc C^\infty(\mc K)$, which is known in the literature as the {\em space of distributions} in $\mc K$. The functions $f \in \mc C^\infty(\mc K)$ are called {\em test functions}. 
We will define two distribution-valued processes associated to $\{\eta_t^n; t \in [0,T]\}$; one with values in $\mc D(\bb T)$ and another one with values in $\mc  D(\bb T)$. It turns out that the spaces $\mc D(\bb T)$, $\mc D(\bb T^2)$ equipped with the weak topology are Polish spaces. For an arbitrary Polish space $\mc E$, we denote by $\mc D([0,T]; \mc E)$ the space of c\`adl\`ag trajectories equipped with the $J_1$-Skorohod topology. Notice that with respect to the $J_1$-Skorohod topology $\mc D([0,T]; \mc E)$ is also a Polish space.

For a given compact space $\mc K$ and $x \in \mc K$, we denote by $\delta_x$ the $\delta$ of Dirac at $x$, that is, the atomic probability measure supported at $\{x\}$. 
Let $\{\mc Y_t^n; t \in [0,T]\}$ denote the $\mc D(\bb T)$-valued process given by
\begin{equation}
\mc Y_t^n(f) = \frac{1}{\sqrt n} \sum_{x \in \bb T_n} \big(\eta_t^n(x)-\tfrac{1}{2}\big) f(x)
\end{equation}
for any $f \in \mc C^\infty(\bb T)$ and any $t \in [0,T]$. The process $\{\mc Y_t^n; t \in [0,T]\}$ is known in the literature as the {\em density fluctuation field} associated to the process $\{\eta_t^n; t \in [0,T]\}$.
This process has been extensively studied, and in particular a scaling limit for it is available.

\begin{proposition}
The process $\{\mc Y_t^n; t \in [0,T]\}$ converges in distribution with respect to the $J_1$-topology of $\mc D([0,T]; \mc D(\bb T))$ to the stationary solution of the infinite-dimensional Ornstein-Uhlenbeck equation
\begin{equation}
\label{ec1.7}
d \mc Y_t = \Delta \mc Y_t dt + \tfrac{1}{\sqrt 2}\nabla d\omega_t,
\end{equation}
where $\{\omega_t(x); x \in \bb T, t \in [0,T]\}$ is a space-time white noise.
\end{proposition}

As far as we understand, this result was proved by the first time in the lecture notes by De Masi, Ianiro, Pellegrinotti and Presutti \cite{D-MIPP}.

Let us introduce a second fluctuation field, this time with values in $\mc D(\bb T^2)$.
Let $\{Q_t^n; t \in [0,T]\}$ be the process with values in $\mc D(\bb T^2)$ defined as
\begin{equation}
Q_t^n(f) = \frac{1}{n} \!\!\sum_{\substack{x,y \in \bb T_n\\x \neq y}} \!\!\big(\eta_t^n(x) -\tfrac{1}{2}\big) \big(\eta_t^n(y) -\tfrac{1}{2}\big) f(x,y)
\end{equation}
for any test function $f \in \mc C^\infty(\bb T^2)$ and any $t \in [0,T]$. Without loss of generality, from now on and up to the end of the article we will assume that any test function $f \in \mc C^\infty(\bb T^2)$ is symmetric. In fact, for any antisymmetric function $f \in \mc C^\infty(\bb T^2)$, $Q_t^n(f) =0$. We call the process $\{Q_t^n; t \in [0,T]\}$ the {\em quadratic} fluctuation field associated to the process $\{\eta_t^n; t \in [0,T]\}$. Our aim will be to obtain the scaling limit of the process $\{Q_t^n; t \in [0,T]\}$ as $n \to \infty$. 

\begin{theorem}
\label{t1}
Let $\{\mc W_t; t \in [0,T]\}$ be the martingale process defined as
\begin{equation}
\mc W_t(f) = \tfrac{1}{\sqrt 2} \int_0^t \big\{ \mc Y_s\big(\partial_2 f(\cdot,x)\big) + \mc Y_s\big(\partial_1 f(x,\cdot)\big)\big\} \omega(dx ds)
\end{equation}
for any test function $f \in \mc C^\infty(\bb T^2)$.
The process $\{Q_t^n; t \in [0,T]\}$ converges in distribution, as $n \to \infty$, to the stationary solution of the equation
\begin{equation}
d Q_t = \Delta Q_t dt + d\mc W_t.
\end{equation}
\end{theorem}

For a given function $f: \bb T_n \to \bb R$, it will be useful to introduce the notation $\mc Y^n(f)$ (without the index $t$) for the function from $\Omega_n$ to $\bb R$ given by
\begin{equation}
\mc Y^n(f) = \frac{1}{\sqrt n} \sum_{x \in \bb T_n} \big(\eta(x)-\tfrac{1}{2}\big) f(x)
\end{equation}
Let us write $\bb T_n^2 = \bb T_n \times \bb T_n$. In the same spirit, for a (symmetric) function $f: \bb T_n^2 \to \bb R$ and for $\eta \in \Omega_n$, we define
\begin{equation}
Q^n(f) = \frac{1}{n} \!\!\sum_{\substack{x,y \in \bb T_n\\x \neq y}} \!\! \big(\eta(x) - \tfrac{1}{2}\big)\big(\eta(y)-\tfrac{1}{2}\big) f(x,y).
\end{equation}

\subsection{A singular fluctuation field}

The main purpose of the introduction of the fluctuation field $\{Q_t; t \in [0,T]\}$, is the study of the {\em quadratic field} $\{\mc A_t; t \in [0,T]\}$ formally defined as follows. For each function $f \in \mc C^\infty(\bb T)$, we denote by $f \otimes \delta$ the distribution in $\bb T^2$ given by
\begin{equation}
\<f \otimes \delta, g\> = \int\limits_{\bb T} f(x) g(x,x) dx.
\end{equation}
In other words, $f \otimes \delta(x,y) = f(x) \delta(x,y)$, where $\delta(x,y)$ represents the uniform measure on the diagonal $\{(x,y) \in \bb T^2; x = y\}$. Then we formally define $\mc A_t(f)$ as
\begin{equation}
\mc A_t(f) = \int_0^t Q_s(f' \otimes \delta) ds.
\end{equation}

Of course, it is not clear at all whether this definition makes any sense. The simplest idea one can use is to approximate the singular object $f' \otimes \delta$ by a sequence of more regular functions. Let us consider the approximation of the identity $\{\iota_\varepsilon; 0<\varepsilon<1\}$ given by
\begin{equation}
\iota_\varepsilon(x,y) = \frac{1}{4\varepsilon^2} \mathbbm{1}_{|x| \leq \varepsilon} \mathbbm{1}_{|y| \leq \varepsilon}.
\end{equation}
This approximation of the identity is not the smoothest one we can use, but it is very convenient computationally.
The following theorem explains how to define the singular process $\{\mc A_t; t \in [0,T]\}$.

\begin{theorem}
\label{t2}
Let $\{\iota_\varepsilon; \varepsilon \in (0,1)\}$ be the approximation of the identity in $\bb T^2$ defined above. Let $\{\mc A_t^\varepsilon; t \in [0,T]\}$ be the $\mc D(\bb T)$-valued process defined as
\begin{equation}
\mc A_t^\varepsilon(f) = \int_0^t Q_s\big((f'\!\otimes \!\delta) \ast \iota_\varepsilon\big) ds
\end{equation}
for any $f \in \mc C^\infty(\bb T)$ and any $t \in [0,T]$.
Then, $\{\mc A_t^\varepsilon; t \in [0,T]\}$ converges in distribution with respect to the uniform topology to a well-defined, $\mc D(\bb T)$-valued process $\{\mc A_t; t \in [0,T]\}$.
\end{theorem}

This theorem was proved in \cite{GJ3}, although the proof there is very different from the proof we will present here. Moreover, the proof we present here has one important advantage: as we will see, it gives a more explicit construction of the process $\{\mc A_t; t \in [0,T]\}$, which allows to obtain various properties of it. For completeness, we present the following convergence result, obtained in \cite{GJ3}:

\begin{proposition}
Let $\{\mc A_t^n; t \in [0,T]\}$ the $\mc D(\bb T)$-valued process defined by
\begin{equation}
\mc A_t^n(f) = \int_0^t \sum_{x \in \bb T_n} \big(\eta_s^n(x)-\tfrac{1}{2}\big)\big(\eta_s^n(x+1)-\tfrac{1}{2}\big) f'(\tfrac{x}{n}) ds
\end{equation}
for any $f \in \mc C^\infty(\bb T)$. The process $\{\mc A_t^n; t \in [0,T]\}$ converges in distribution as $n \to \infty$,  to the process $\{\mc A_t; t \in [0,T]\}$ defined above.
\end{proposition}

Notice that concatenating intervals of size $T$, we can assume that $\{\mc A_t; t \geq 0\}$ is a well-defined process.
As we mentioned in the introduction, our construction allows to obtain some properties of the field $\{\mc A_t; t \geq 0\}$. The short-time properties of $\{\mc A_t  ;  t\geq 0\}$ are given by the following theorem:
\begin{theorem}
\label{t4.5}
As $\varepsilon \to 0$, for each $f \in \mc C^\infty(\bb T)$ the field $\{\varepsilon^{-3/4} \mc A_{\varepsilon t}(f); t \geq 0\}$ converges in the uniform topology to a stationary Gaussian process $\{\bb A_t(f); t \geq 0\}$ of covariances
\begin{equation}
E\big[\bb A_t(f)\bb A_s(f)\big] = \tfrac{\kappa}{8}\<f,-\Delta f\> \big\{ t^{3/2}+s^{3/2}-|t-s|^{3/2}\big\},
\end{equation}
where
\begin{equation}
\kappa = \frac{1}{8\pi^4}\int\limits_{\bb R} \Big( \frac{1-e^{-\pi^2 x^2}}{x^2}\Big)^2 dx.
\end{equation}
\end{theorem}

Notice that this result is very similar in spirit to Theorem 2.5 of \cite{GJ3}. As far as we understand, this result has not been predicted in the literature. For large times $t$, we only know the limiting variance of $\mc A_t(f)$:

\begin{theorem}
\label{t2.6} There exists a compact operator $\mc K: L^2(\bb T) \to L^2(\bb T)$ such that
\begin{equation}
\lim_{t \to \infty} \frac{1}{t} \bb E\big[\mc A_t(f)^2 \big] = \tfrac{1}{4} \< f, -\tfrac{1}{\sqrt{2 \pi}} (-\Delta)^{\frac{1}{4}} f + \mc K f\>.
\end{equation}
\end{theorem}

\section{Tightness}
\label{s2}
We will prove Theorem \ref{t1} using the standard three-steps method to get convergence in distribution of stochastic processes, namely, we first prove tightness of the sequence of processes in a suitable topology, then we deduce some properties of the possible limit points using the approximating processes, and then we show that these aforementioned properties characterize the limit point in a unique way. In this section we perform the first step, that is, we prove tightness, and we prepare the ground for the other steps by introducing a bunch of martingales.

\subsection{The associated martingales}
We start recalling the following well-known fact about continuous-time Markov chains.
Let $f: \Omega_n \times [0,T] \to \bb R$ be a smooth function on the time variable. Then, the process
\begin{equation}
\exp\Big\{f_t(\eta_t^n) - f_0(\eta_0^n) - \int_0^t e^{-f_s(\eta_s^n)}\big(\partial_s+ L_n\big) e^{f_s(\eta_s^n)} ds\Big\}
\end{equation}
is a positive martingale with unit expectation. We will apply this formula for the function $\theta Q^n(f)$, where $\theta \in \bb R$ and $Q^n(f)$ is defined above. For each  $\theta \in \bb R$ and each $f \in \mc C^\infty$, let $\{\mc M_t^{\theta,n}(f); t \in [0,T]\}$ denote the martingale given by
\begin{equation}
\mc M_t^{\theta,n}(f) = \exp\Big\{ \theta \big(Q^n_t(f) - Q_t^n(f)\big) - \int_0^t e^{-\theta Q_s^n(f)} L_n e^{\theta Q_s^n(f)} ds \Big\}
\end{equation}
for any $t \in [0,T]$.
Another simple observation is that for any $\ell \in \bb N$, the process
\begin{equation}
\label{ec2.3}
\frac{\partial^{\ell}}{\partial \theta^\ell} \mc M_t^{\theta,n}(f) \Big|_{\theta = 0}
\end{equation}
is a martingale. Let us compute $\mc M_t^{\theta,n} (f)$ in a  more explicit way. Notice that for any function $f: \Omega_n \to \bb R$,
\begin{equation}
e^{-f} L_n e^{f} = \sum_{x \sim y} n^2\big(e^{\nabla_{\!x,y}f}-1\big).
\end{equation}
Notice as well that 
\begin{equation}
\label{ec2.5}
\nabla_{x,y} Q^n(f) = \frac{1}{n} \!\!\sum_{z \neq x,y}\!\!\big(\eta(z)-\tfrac{1}{2}\big) \big(f(z,y)-f(z,x)\big)\big(\eta(x)-\eta(y)\big).
\end{equation}
This term will appear repeatedly in what follows, so we will give it a name. Let $\xi^n(f;x,y): \Omega_n \to \bb R$ denote the function
\begin{equation}
\label{ec2.6}
\xi^n(f;x,y)(\eta) = \sum_{z \neq x,y} \big(\eta(z)-\tfrac{1}{2}\big) \big(f(z,y)-f(z,x)\big)\big(\eta(x)-\eta(y)\big).
\end{equation}
If $y = x + \smash{\frac{1}{n}}$, we will  write $\xi^n(f;x)$ instead of $\xi^n(f;x,y)$.
We will write $\xi_t^n(f;x,y) = \xi^n(f;x,y)(\eta_t^n)$, $\xi_t^n(f;x) = \xi^n(f;x)(\eta_t^n)$ as well. With this notation, we can write \eqref{ec2.5} in the more compact form
\begin{equation}
\nabla_{x,y} Q^n(f) = \frac{1}{n} \xi^n(f;x,y).
\end{equation}
When $y=x+\smash{\frac{1}{n}}$, the difference $f(z,y) - f(z,x)$ is of order $\smash{\frac{1}{n}}$. In fact, since $f \in \mc C^\infty$, this difference is approximated by $\smash{\frac{1}{n}}\partial_2 f(z,x)$, with an error term of order $\smash{\frac{1}{n^2}}$. In particular, we have the {\em a priori} bound 
\begin{equation}
\big|\xi^n_t(f;x)\big| \leq c_1(f)
\end{equation}
for any $n \in \bb N$, any $x \in \bb T_n$ and any $t \in [0,T]$. Since the measure $\nu_{\frac{1}{2}}$ is of product form, we have a better bound for the second moment of $\xi^n_t(f;x)$:
\begin{equation}
\label{ec2.9}
\bb E_n\big[\xi_t^n(f;x)^2\big] \leq \frac{c_2(f)}{n}
\end{equation}
for any $n \in \bb N$, any $x \in \bb T_n$ and any $t \in [0,T]$. We point out here that the {\em a priori} bound will be very useful, since the term $\xi_t^n(f;x)$ will actually appear in a double exponential, and therefore moment bounds will not be as useful.
Going back to the computation of the martingale, we see that 
\begin{equation}
\mc M_t^{\theta,n}(f) = \exp\Big\{\theta\big(Q^n_t(f)- Q^n_0(f)\big) -\int_0^t \!\!\sum_{x \in \bb T_n}\!\! n^2 \big(e^{\frac{\theta}{n} \xi_s^n(f;x)}-1\big) ds\Big\}.
\end{equation}
Recall that taking derivatives of $\mc M_t^{\theta,n}$ with respect to $\theta$ we can find other martingales associated to the process $\{\eta_t^n; t \in [0,T]\}$. Taking $\ell =1$ in \eqref{ec2.3}, we see that the process $\{\mc W_t^n(f); t \in [0,T]\}$ given by
\begin{equation}
\label{ec2.11}
\mc W_t^n(f) = Q_t^n(f)-Q_0^n(f) - \int_0^t \sum_{x \in \bb T_n} n \xi_s^n(f;x) ds
\end{equation}
is a martingale.
For $f \in \mc C^\infty(\bb T^2)$, let us define $\Delta_n f : \bb T_n \times \bb T_n \to \bb R$ as
\begin{equation}
\Delta_n f(x,z) = \sum_{\substack{y \in \bb T_n\\y \sim x}} n^2\big(f(y,z)-f(x,z)\big) +\sum_{\substack{w \in \bb T_n\\w \sim z}} n^2 \big(f(x,w)-f(x,z)\big).
\end{equation}
In other words, $\Delta_n f$ is a discrete version of the Laplacian $\Delta f = \partial_{11} f + \partial_{22} f$. For functions $f \in \mc C^\infty(\bb T)$, we define
\begin{equation}
\Delta_n f(x) = n^2 \!\!\sum_{\substack{y \in \bb T_n\\ y \sim x}} \big(f(y) -f(x)\big),
\end{equation} 
which corresponds to a discrete version of the Laplacian (actually the second derivative!) $\Delta f$. For functions $f$ defined only in $\bb T_n$, we adopt the same notation and definition for $\Delta_n f$.
After some computations, we see that
\begin{equation}
\label{ec2.14}
\begin{split}
\sum_{x \in \bb T_n} n\xi^n(f;x) 
		&= \frac{1}{n} \!\!\sum_{x,y \in \bb T_n}\!\! \big(\eta(x) -\tfrac{1}{2}\big) \big(\eta(z) -\tfrac{1}{2}\big) \Delta_n f(x,z)\\
		&\quad -\frac{1}{2n} \sum_{x \in \bb T_n}\big(\eta(x)-\eta(x+\tfrac{1}{n})\big)^2 \Delta_n \diag(f)(x),
\end{split} 
\end{equation}
where the function $\diag(f): \bb T \to \bb R$ is defined as $\diag(f)(x) = f(x,x)$ for any $x \in \bb T$, that is, $\diag(f)$ is the value of $f$ on the {\em diagonal} $\{x=y\}$. Notice that the first term on the right-hand side of \eqref{ec2.14} can be written as $Q^n(\Delta_n f)$. With respect to the measure $\nu_{\frac{1}{2}}$, the expectation of the second term on the right-hand side of \eqref{ec2.14} is equal to $0$, since the function $\Delta_n \diag(f)$ has mean $0$ with respect to the counting measure on $\bb T_n$. It is also uniformly bounded in $n$ and $\eta$, thanks to the smoothness of $f$. Moreover, its variance with respect to $\nu_{\frac{1}{2}}$ vanishes as $n \to \infty$. In other words, only the first term on the right-hand side of \eqref{ec2.14} will be relevant when considering scaling limits.

The term $(\eta(x)-\eta(x+\smash{\frac{1}{n}}))^2$ will appear repeatedly in what follows, so it is convenient to give it a name. Let us define
\begin{equation}
c_x(\eta) = \big(\eta(x)-\eta(x+\tfrac{1}{n})\big)^2
\end{equation}
and notice that $0 \leq c_x(\eta) \leq 1$ for any $n \in \bb N$, any $x \in \bb T_n$ and any $\eta \in \Omega_n$.
Taking $\ell =2$ in \eqref{ec2.3}, we see that the process 
\begin{equation}
\mc W_t^n(f) ^2 - \int_0^t \sum_{x \in \bb T_n} \xi^n_s(f;x)^2 ds
\end{equation}
is also a martingale. In other words, the {\em quadratic variation} of the martingale $\{\mc W_t^n(f); t \in [0,T]\}$ is given by
\begin{equation}
\<\mc W_t^n(f)\> = \int_0^t \sum_{x \in \bb T_n} \xi^n_s(f;x)^2 ds.
\end{equation}
Notice that the moment bound \eqref{ec2.9} implies a bound of the form $\bb E_n\<\mc W_t^n(f)\> \leq c(f) t$ for the quadratic variation process, and therefore a moment bound of the form $\bb E_n[\mc W_t^n(f)^2] \leq c(f) t$ for the martingale process $\{\mc W_t^n(f); t \in [0,T]\}$. This observation will be relevant when showing tightness of the processes $\{Q_t^n; t \in [0,T]\}$.

\subsection{Tightness}

In this section we prove tightness of the sequence of processes $\{Q^n_t; t \in [0,T]\}_{n \in \bb N}$. The first step is to reduce the problem from distribution-valued processes to real-valued processes. This is done throught the so-called {\em Mitoma's criterion}.

\begin{proposition}[Mitoma's criterion \cite{Mit}]
The sequence of $\mc D(\bb T^2)$-valued processes $\{Q^n_t; t \in [0,T]\}_{n \in \bb N}$ is tight in $\mc D([0,T]; \mc D(\bb T^2))$ if and only if the sequence $\{Q_t^n(f); t \in [0,T]\}_{n \in \bb N}$ is tight in $\mc D([0,T]; \bb R)$ for any $f \in \mc C^\infty(\bb T^2)$. Moreover, if any limit point of $\{Q_t^n(f); t \in [0,T]\}_{n \in \bb N}$ is supported on continuous, real-valued trajectories for any $f \in \mc C^\infty(\bb T^2)$, then any limit point of $\{Q_t^n; t \in [0,T]\}_{n \in \bb N}$ is supported on continuous trajectories on $\mc D(\bb T^2)$.
\end{proposition}

Applying this criterion, we see that it is enough to prove tightness for the process
\begin{equation}
\label{ec2.17}
Q_t^n(f) = Q_0^n(f) + \int_0^t \sum_{x \in \bb T_n} n \xi_s^n(f;x) ds  +\mc W_t^n(f)
\end{equation}
for each $f \in \mc C^\infty(\bb T^2)$. The decomposition \eqref{ec2.17} is also a declaration of intentions: we will not prove tightness directly for the process $\{Q_t^n(f); t \in [0,T]\}$, but for each one of the process appearing on the right-hand side of \eqref{ec2.17}. The simplest one is the sequence $\{Q_0^n(f); n \in \bb N\}$ of real-valued random variables. In fact, taking characteristic functions, it is easy to see that $Q_0^n(f)$ converges to a Gaussian random variable of mean $0$ and variance $\smash{\frac{1}{16}} \iint f(x,y)^2 dx dy$. Notice that this is also true for the sequence $\{Q_t^n(f); n \in \bb N\}$ for any time $t \in [0,T]$. Since any convergent sequence is tight, we are done with this term.

Next in line is the integral term 
\begin{equation}
\mc I_t^n(f) = \int_0^t \sum_{x \in \bb T_n} n \xi_s^n(f;x) ds.
\end{equation}
There is a very simple criterion that applies in this situation (see Theorem 2.3 of \cite{FPV} for example).

\begin{proposition}
\label{p2.2}
The sequence of processes $\{\mc I_t^n(f); t \in [0,T]\}_{n \in \bb N}$ is tight with respect to the uniform topology of $\mc C([0,T]; \bb R)$ if
\begin{equation}
\sup_{n \in \bb N}\sup_{0 \leq t \leq T} \bb E_n\Big[\Big( \sum_{x \in  \bb T_n} n\xi_t^n(f;x)\Big)^2\Big] <\infty.
\end{equation} 
\end{proposition}
By the stationarity of the process $\{\eta_t^n; t \in [0,T]\}$, it is enough to get a bound for 
\begin{equation}
E_{\frac{1}{2}}\Big[\Big(\sum_{x \in \bb T_n} n \xi^n(f;x)\Big)^2\Big].
\end{equation}
Recall decomposition \eqref{ec2.14} and the comments thereafter. Since $f$ is smooth, we see that there exists a constant $C(f)$ which does not depend on $n$ such that
\begin{equation}
E_{\frac{1}{2}}\Big[\Big(\frac{1}{2n} \sum_{x \in \bb T_n} c_x(\eta) \Delta_n\diag(f)(x)\Big)^2\Big] \leq C(f)
\end{equation}
for any $n \in \bb N$. In the other hand,
\begin{equation}
E_{\frac{1}{2}}\Big[ \Big(  \frac{1}{n} \!\!\sum_{x,y \in \bb T_n}\!\! \big(\eta(x) -\tfrac{1}{2}\big) \big(\eta(z) -\tfrac{1}{2}\big) \Delta_n f(x,z)\Big)^2\Big] \leq \frac{1}{n^2} \sum_{x,z \in \bb T_n} \big( \Delta_n f(x,z)\big)^2,
\end{equation}
and due to the smoothness of $f$, the right-hand side of this inequality can also be bounded by a (maybe different) constant $C(f)$ which does not depend on $n$. We conclude that the sequence of processes $\{\mc I_t^n(f); t \in [0,T]\}_{n \in \bb N}$ is tight with respect to the uniform topology of $\mc C([0,T]; \bb R)$.

Now it is the turn of the martingale term $\{\mc W_t^n(f); t \in [0,T]\}$. We use the following well-known criterion.

\begin{proposition}
The sequence of real-valued processes $\{\mc W_t^n(f); t \in [0,T]\}_{n \in \bb N}$ is tight in $\mc D([0,T]; \bb R)$ if
\begin{equation}
\label{ec2.23}
\sup_{n \in \bb N} \int_0^T \bb E_n\Big[\Big(\sum_{x \in \bb T_n} \xi_t^n(f;x)^2\Big)^2\Big] dt <\infty.
\end{equation}
\end{proposition}

Let us show that this supremum is finite. First we recall a simple version of Burkholder's inequality.

\begin{proposition}
For any $p \geq 1$ there exists a universal constant $\kappa_p$ such that
\begin{equation}
E_{\frac{1}{2}} \Big[\Big|\sum_{z \in A} \big(\eta(z)-\tfrac{1}{2}\big) f_z\Big|^{2p}\Big] 
		\leq \kappa_p \Big(\sum_{z \in A} f_z^2\Big)^p
\end{equation}
for any $A \subseteq \bb T_n$ and any $f: A \to \bb R$.
\end{proposition}

Notice that the inequality above is actually an identity when $p=1$ and $\kappa_1=\frac{1}{4}$. By stationarity, we just need to show that
\begin{equation}
\sup_{n \in \bb N} E_{\frac{1}{2}}\Big[\Big(\sum_{x \in \bb T_n} \xi^n(f;x)^2\Big)^2\Big] < \infty.
\end{equation}
Since there are no cancellations between the $\xi^n(f;x)$ terms, using the crude estimate $(a_1+\dots a_\ell)^2 \leq \ell (a_1^2+\dots+a_\ell)$ should not be really bad. Therefore, we need to show that
\begin{equation}
\sup_{n \in \bb N} n \sum_{x \in \bb T_n}  E_{\frac{1}{2}}\big[\xi^n(f;x)^4\big] < \infty.
\end{equation}
Bounding above by $1$ the term $c_x(\eta)$ in $\xi^n(f;x)^2$, we are exactly on the setup of Burkholder's inequality for $p=2$. Therefore, we have the bound 
\begin{equation}
n \sum_{x \in \bb T_n} E_{\frac{1}{2}}\big[ \xi^n(f;x)^4\big] \leq \kappa_2 n \sum_{x \in \bb T_n} \Big(\sum_{z \in \bb T_n} \big(f(z,x+\tfrac{1}{n})-f(z,x)\big)^2\Big)^2.
\end{equation}
Due to the smoothness of $f$, the right-hand side of this inequality converges, as $n \to \infty$, to
\begin{equation}
\int \Big(\int \partial_2 f(z,x)dz\Big)^2 dx,
\end{equation}
which shows that the supremum in \eqref{ec2.23} is finite. We conclude that the martingale sequence $\{\mc W_t^n(f); t \in [0,T]\}_{n \in \bb N}$ is tight in $\mc D([0,T]; \bb R)$.

Recapitulating what we have accomplished up to here in this section, we have proved tightness of $\{Q_t^n; t \in [0,T]\}_{n \in \bb N}$ by combining Mitoma's criterion with the tightness of the processes appearing in the decomposition \eqref{ec2.17}.
Here we just remark that we do not know whether the limit points of $\{\mc W_t^n(f); t \in [0,T]\}_{n \in \bb N}$ are continuous, and in particular, we do not know whether the limit points of $\{Q_t^n; t \in [0,T]\}$ are concentrated on continuous, $\mc D(\bb T^2)$-valued trajectories. We will address this question in the following section, when discussing the characterization of such limit points.

\section{Characterization of limit points}
\label{s3}
In the previous section, we have showed tightness of the sequence $\{Q_t^n; t \in [0,T]\}_{n \in \bb N}$. Let $\{Q_t; t \in [0,T]\}$ be one of such limit points. For ease of notation, the superscript $n$ will denote along this section, the subsequence for which the process $\{Q_t^n; t \in [0,T]\}$ converges to $\{Q_t; t \in [0,T]\}$. Without loss of generality, we can assume that the process $\{\mc W_t^n(f); t \in [0,T]\}$ converges to a process $\{\mc W_t(f); t \in [0,T]\}$, along the subsequence $n$, for any $f \in \mc C^\infty(\bb T^2)$. Interpreting $\{\<\mc W_t^n(f)\>; t \in [0,T]\}$ as an integral process, we see that the estimates used to prove tightness of $\{\mc W_t^n(f); t \in [0,T]\}$ are exactly the ones needed to apply the tightness criterion \ref{p2.2} to $\{\<\mc W_t^n(f)\>; t \in [0,T]\}$. We leave the details to the reader. Taking a further sub-subsequence if necessary, we can assume that along $n$, the process $\{\<\mc W_t^n(f)\>; t \in [0,T]\}$ converges in distribution to some process $\{\<\mc W_t(f)\>; t \in [0,T]\}$
Without further comments, taking sub-subsequences if necessary, we will assume that any process we need to define is convergent along the subsequence $n$. Notice that based solely on convergence in distribution, we can not argue that $\{\<\mc W_t(f)\>; t \in [0,T]\}$ is the quadratic variation of $\{\mc W_t(f); t \in [0,T]\}$. In fact, we do not even know whether $\{\mc W_t(f); t \in [0,T]\}$ is a martingale.

\subsection{The exponential martingales}
\label{s3.1}
In order to address the questions raised in the previous section, it will be more convenient to work with the exponential martingales $\{\mc M_t^{\theta, n}(f); t \in [0,T]\}$. For the moment we do not know whether these martingales are tight, and therefore we can not say anything about convergence. The following simple Taylor estimate will be very useful.

\begin{proposition}
For any $u \in \bb R$ and any $\ell \in \bb N$,
\begin{equation}
\Big|e^u -\sum_{i=0}^{\ell-1} \frac{u^i}{i!}\Big| \leq \frac{1}{\ell!} |u|^\ell e^{|u|}.
\end{equation}
\end{proposition}
Fix $f \in \mc C^\infty(\bb T^2)$ and $\theta \in \bb R$. We will use this estimate for $u = \frac{\theta}{n} \xi^n_s(f;x)$ and $\ell =4$. Recall that $\xi^n(f;x) \leq c_1(f)$. We have that 
\begin{equation}
\Big|e^{\frac{\theta}{n}\xi_s^n(f;x)}-1-\tfrac{\theta}{n} \xi^n_s(f;x) -\tfrac{\theta^2}{2n^2} \xi^n_s(f;x)^2 -\tfrac{\theta^3}{6n^3} \xi^n_s(f;x)^3\Big| 
		\leq  \tfrac{\theta^4}{24 n^4} c_1(f)^4 e^{\theta c_1(f)}.
\end{equation}
Therefore, we can write
\begin{equation}
\label{ec2.31}
\mc M_t^{\theta,n}(f) = \exp\Big\{\theta \mc W_t^n(f) -\tfrac{\theta^2}{2} \<\mc W_t^n(f)\>- \int_0^t \tfrac{\theta^3}{6n} \sum_{x \in \bb N} \xi^n_s(f;x)^3ds + \mc R_t^{\theta,n}(f)\Big\},
\end{equation}
where the error term $\mc R_t^{\theta,n}(f)$ satisfies 
\begin{equation}
\big| \mc R_t^{\theta,n}\big| \leq \tfrac{\theta^4c_1(f)^4}{24 n} e^{\theta c_1(f)},
\end{equation}
and in particular this error term goes to $0$ uniformly in $n$. Notice that
\begin{equation}
\Big| \frac{1}{n} \sum_{x \in \bb T_n} \xi_s^n(f;x)^3\Big| \leq c_1(f)^3,
\end{equation}
\begin{equation}
\bb E_n \Big[\Big|\frac{1}{n} \sum_{x \in \bb T_n} \xi_s^n(f;x)^3\Big|\Big] 
		\leq \frac{c_1(f) c_2(f)}{n}
\end{equation}
and we see that the cubic term in \eqref{ec2.31} is uniformly bounded in $n$ and it converges to $0$ in $\bb P_n$-probability as $n \to \infty$. The uniform bound is relevant, because this cubic term is on top of an exponential in \eqref{ec2.31}. Looking into \eqref{ec2.31}, we have just showed that along the subsequence $n$, for any fixed $t \in [0,T]$ and any $f \in \mc C^\infty$, the {\em random variable} $\mc M_t^{\theta,n}(f)$ converges in distribution to 
\begin{equation}
\mc M_t^\theta = \exp\{ \theta \mc W_t(f) - \tfrac{1}{2}\theta^2 \<\mc W_t(f)\>\}.
\end{equation}
Notice that although we have not proved that the convergence holds at the level of processes, we do know that  $\{\mc M_t^\theta; t \in [0,T]\}$ is a well-defined process with trajectories in $\mc D([0,T]; \bb R)$. In order to prove that $\{\mc M_t^\theta; t \in [0,T]\}$ is a martingale, it is enough to show that the sequence $\{\mc M_T^{\theta, n}; n \in \bb N\}$ is uniformly integrable. It is here where we will take full advantage of the decomposition \eqref{ec2.31}. In fact, the error term and the cubic term in \eqref{ec2.31} are uniformly bounded in $n$ by a deterministic constant. Therefore, we can neglect them. Moreover, the term $\<\mc W_t^n(f)\>$ is non-negative and it appears with a minus sign. Therefore, we can also neglect it. We are left to prove the uniform integrability of the sequence $\{\exp\{\theta \mc W_T^n(f)\}; n \in \bb N\}$. The simplest criterion for uniform integrability is a uniform $L^p(\bb P_n)$-bound for some $p>1$. In other words, we want to estimate $\bb E_n[\exp\{p\theta \mc W_t^n(f)\}]$ for some $p>1$. Recall that $\bb E_n[\mc M_T^{\theta,n}]=1$. We will use the elementary estimate 
\begin{equation}
E[e^X] \leq E[e^{2(X-Y)}]^{1/2}E[e^{2Y}]^{1/2}
\end{equation}
for $X = p\theta \mc W_T^n$ and
\begin{equation}
Y = p^2\theta^2 \<\mc W_T^n\> + \int_0^T \frac{2p^3\theta^3}{3n} \sum_{x \in \bb T_n} \xi^n_t(f;x)^3 dt + \frac{1}{2}\mc R_T^{2p\theta,n}(f).
\end{equation}
With these choices, we get the bound 
\begin{equation}
\label{ec2.38}
\bb E_n[\exp\{p\theta \mc W_T^n(f)\}] \leq \bb E_n[\exp\{ 2 Y\}]^{\frac{1}{2}}.
\end{equation}
We can again neglect the cubic term and the error term in \eqref{ec2.38}, since they are uniformly bounded in $n$ by a deterministic constant. Therefore, uniform integrability of $\{\mc M_T^{\theta,n}; n \in \bb N\}$ will be proved if we can show that
\begin{equation}
\sup_{n \in \bb N} \bb E_n[\exp\{2p^2\theta^2 \<\mc W_T^n(f)\>\}] < \infty.
\end{equation}
Recall the definition of $\<\mc W_t^n(f)\>$ and rewrite it as
\begin{equation}
\<\mc W_T^n(f)\> = \frac{1}{T}\int_0^T \sum_{x \in \bb T_n} T\xi^n_t(f;x)^2 dt.
\end{equation}
By the convexity of the exponential function,
\begin{equation}
\bb E_n[\exp\{2p^2\theta^2 \<\mc W_T^n(f)\>\}] \leq \frac{1}{T} \int_0^T \bb E_n\Big[\exp\{2p^2\theta^2T \sum_{x \in \bb T_n} \xi^n_t(f;x)^2 \}\Big] dt.
\end{equation}
Therefore, we just need to show that
\begin{equation}
\label{ec2.42}
\sup_{n \in \bb N} E_{\frac{1}{2}} \Big[ \exp\{2p^2\theta^2 T \sum_{x \in \bb T_n} \xi^n(f;x)^2\}\Big] <\infty.
\end{equation}
At this point, we need something stronger than Burkholder's inequality in order to bound this expectation. We will use {\em Hoeffding's inequality}.

\begin{proposition}[Hoeffding's inequality]
There exist constants $C_H$, $c_H$ such that  for any $n \in \bb N$, for any $A \subseteq \bb T_n$ and for any $f: A \to \bb R$,
\begin{equation}
\label{ec2.43}
P_{\frac{1}{2}}\Big(\Big| \sum_{x \in A} \big(\eta(z) -\tfrac{1}{2}\big) f(z)\Big| \geq \lambda\Big) \leq C_H \exp\Big\{ -\frac{c_H \lambda^2}{\sum_{z \in A} f(z)^2}\Big\}.
\end{equation}
\end{proposition}
Let us go back to \eqref{ec2.42}. As we pointed out before, there are no cancellations we can take advantage when adding up the terms $\xi^n(f;x)^2$. Therefore, we use the crude exponential H\"older estimate $E[\exp\{\sum_i X_i\}] \leq \prod_i E[\exp\{nX_i\}]^{1/n}$ to get the bound
\begin{equation}
E_{\frac{1}{2}} \Big[ \exp\{2p^2\theta^2 T \sum_{x \in \bb T_n} \xi^n(f;x)^2\}\Big]
		\leq \prod_{x \in \bb T_n} E_{\frac{1}{2}} \big[ \exp\{2p^2\theta^2 T n \xi^n(f;x)^2\}\big]^{\frac{1}{n}}.
\end{equation}
To simplify the notation, let us write $\beta = 2p^2 \theta^2 T$. Bounding $c_x(\eta)$ by one, we get the estimate
\begin{equation}
\label{ec2.45}
E_{\frac{1}{2}} \big[ \exp\{\beta n \xi^n(f;x)^2\}\big] \leq
		E_{\frac{1}{2}} \Big[ \exp\Big\{\beta n \Big(\sum_{z \in A_x} \big(\eta(z)-\tfrac{1}{2}\big)g(z,x)\Big)^2\Big\}\Big],
\end{equation}
where $A_x = \bb T_n \setminus \{x,x+\smash{\frac{1}{n}}\}$ and $g(z,x)=f(z,x) - f(z,x+\smash{\frac{1}{n}})$. We are almost at the setting of Hoeffding's inequality. We just need the following simple observation. For any non-negative random variable $X$ and any regular function $f$,
\begin{equation}
E[f(X)] = -f(0)+ \int_0^\infty f'(t) P(X \geq t) dt.
\end{equation}
Therefore, the right-hand side of \eqref{ec2.45} is bounded above by
\begin{equation}
		\int_0^\infty e^t P_{\frac{1}{2}}\Big( \Big|\sum_{z \in A_x} \big(\eta(z)-\tfrac{1}{2}\big)g(z,x)\Big| \geq \sqrt{\frac{\smash{t}}{\smash{\beta n} }}\Big) dt.
\end{equation}
Finally we are in position to use Hoeffding's inequality. Taking $\lambda = \sqrt{\frac{t}{\beta n}}$ in \eqref{ec2.43}, the integral above is bounded by
\begin{equation}
\int_0^\infty C_H\exp\Big\{ -t\Big( \frac{c_H}{\beta n \sum_{z \in A_x} g(z,x)^2}-1\Big)\Big\} dt.
\end{equation}
By the smoothness of $f$, there is a constant $c_2(f)$ such that 
\begin{equation}
n \sum_{z \in A_x} g(z,x)^2 \leq c_2(f)
\end{equation}
for any $n \in \bb N$ and any $x \in \bb T_n$. Therefore, for any $\beta \leq \frac{c_H}{c_2(f)}$, 
\begin{equation}
\int_0^\infty C_H\exp\Big\{ -t\Big( \frac{c_H}{\beta n \sum_{z \in A_x} g(z,x)^2}-1\Big)\Big\} dt
		\leq C_H\Big\{\frac{c_H}{c_H-\beta c_2(f)}-1\Big\}
\end{equation}
and we conclude that
\begin{equation}
E_{\frac{1}{2}} \Big[ \exp\{2p^2\theta^2 T \sum_{x \in \bb T_n} \xi^n(f;x)^2\}\Big]
		\leq C_H\Big\{\frac{c_H}{c_H-2p^2\theta^2 T c_2(f)}-1\Big\}
\end{equation}
for any $\theta \in \bb R$ small enough. We conclude that the sequence $\{\mc M_T^{\theta,n}; n \in \bb N\}$ is uniformly integrable for any $\theta \in \bb R$ satisfying
\begin{equation}
|\theta| \leq \sqrt{\frac{c_H}{2Tc_2(f)}}.
\end{equation}

We summarize what we have shown above in the following theorem.

\begin{theorem}
\label{t3.3}
For any function $f \in \mc C^\infty$, there is a (strictly positive) constant $\theta(f)$ such that the process $\{\mc M_t^\theta(f); t \in [0,T]\}$ given by
\begin{equation}
\mc M_t^\theta(f) = \exp\{ \theta \mc W_t(f) - \tfrac{1}{2} \theta^2 \<\mc W_t(f)\>\} 
\end{equation}
is a martingale for any $\theta \in \bb R$ such that $|\theta| \leq \theta(f)$.
\end{theorem}

Notice that this theorem implies that the processes
\begin{equation}
\Big\{\frac{\partial^{\ell}}{\partial \theta^\ell} \mc M_t^{\theta}(f) \Big|_{\theta = 0}; t \in [0,T]\Big\}
\end{equation}
are martingales for any $\ell \in \bb N_0$. In particular, considering the cases $\ell = 1,2$ we get the following result.

\begin{corollary}
\label{c3.4}
For any function $f \in \mc C^\infty$, the process $\{\mc W_t(f); t \in [0,T]\}$ is a martingale of quadratic variation $\{\<\mc W_t(f)\>; t \in [0,T]\}$.
\end{corollary}

Considering the case $\ell =4$, it is possible to obtain the following bound:
\begin{equation}
\bb E_n\big[\mc W_t(f)^4\big] \leq 125 \bb E_n\big[\<\mc W_t(f)\>^2\big].
\end{equation}
Using a crude Cauchy-Schwarz estimate to bound $\bb E_n[\<\mc W_t(f)\>^2]$, we get a bound of the form $\bb E_n[\mc W_t(f)^4] \leq c(f) t^2$ for some constant depending on $f$. By stationary and the classical Kolmogorov-Centsov continuity criterion, we conclude that the process $\{\mc W_t(f); t \in [0,T]\}$ has continuous trajectories, which answers the question raised in the previous section about the continuity of $\{Q_t; t \in [0,T]\}$. We state this as a corollary.

\begin{corollary}
\label{c3.5}
Any limit point of the sequence of $\mc D(\bb T^2)$-valued processes $\{Q_t^n; t \in [0,T]\}_{n \in \bb N}$ is concentrated on continuous, $\mc D(\bb T^2)$-valued tractories.
\end{corollary}

\subsection{Characterization of the martingale processes}
\label{s3.2}

Our  next goal will be to characterize the martingale $\{\mc W_t(f); t \in [0,T]\}$ and its quadratic variation  in terms of the processes $\{Q_t; t \in [0,T]\}$ and $\{\mc Y_t; t \in [0,T]\}$. Recall the definition of the process $\{\mc Y_t; t \in [0,T]\}$ as the scaling limit of the density fluctuation fields $\{\mc Y_t^n; t \in [0,T]\}_{n \in \bb N}$. By \eqref{ec2.14},
\begin{equation}
\begin{split}
\sum_{x \in \bb T_n} n\xi^n_s(f;x) 
		&= \frac{1}{n} \!\!\sum_{x,y \in \bb T_n}\!\! \big(\eta_s^n(x) -\tfrac{1}{2}\big) \big(\eta_s^n(z) -\tfrac{1}{2}\big) \Delta_n f(x,z)\\
		&\quad -\frac{1}{2n} \sum_{x \in \bb T_n}\big(\eta_s^n(x)-\eta_s^n(x+\tfrac{1}{n})\big)^2 \Delta_n \diag(f)(x).
\end{split} 
\end{equation}
The second sum on the right-hand side of this identity converges to $0$ in $L^2(\bb P_n)$. The first sum is equal to $Q_s^n( \Delta f)$ plus an error term which also goes to $0$ in $L^2(\bb P_n)$. Taking the limit in \eqref{ec2.11} through the subsequence $n$, we get the identity
\begin{equation}
\mc W_t(f) = Q_t(f) - Q_0(f) - \int_0^t Q_s(\Delta f) dt,
\end{equation}
valid for any $f \in \mc C^\infty$ and any $t \in [0,T]$. For $x \in \bb T$, $n \in \bb N$ and $f \in \mc C^\infty$, let $f_x^n: \bb T \to \bb R$ be defined as
\begin{equation}
f_x^n(z) = n\big(f(z,x+\tfrac{1}{n})-f(z,x)\big).
\end{equation}
Notice that $f_x^n$ is a discrete approximation of the partial derivative $\partial_2f(z,x)$. The function $f_x^n$ provides a very compact way of identify $\xi^n_t(f;x)$ in terms of $\mc Y_t^n$. In fact, looking back to Definition \eqref{ec2.6}, we see that 
\begin{multline}
\xi_t^n(f;x) -\frac{1}{\sqrt n}\big(\eta_t^n(x)-\eta_t^n(x+\tfrac{1}{n})\big) \mc Y_t^n(f_x^n)=\\
		=\big\{\big(\eta_t^n(x)-\tfrac{1}{2}\big)\big(\eta_t^n(x+\tfrac{1}{n})-\tfrac{1}{2}\big) -\tfrac{1}{4}\big\}\big(f(x+\tfrac{1}{n},x+\tfrac{1}{n})+f(x,x)-f(x,x+\tfrac{1}{n})\big).
\end{multline}
In particular, there exists a constant $c_3(f)$ such that 
\begin{equation}
\Big|\xi_t^n(f;x) -\frac{1}{\sqrt n}\big(\eta_t^n(x)-\eta_t^n(x+\tfrac{1}{n})\big) \mc Y_t^n(f_x^n)\Big| \leq \frac{c_3(f)}{n^2}.
\end{equation}
for any $n \in \bb N$, any $x \in \bb T_n$ and any $t \in [0,T]$. Using the simple identity $a^2-b^2 = 2a(a-b)-(a-b)^2$, we see that 
\begin{equation}
\Big|\sum_{x \in \bb T_n} \big\{\xi_t^n(f;x)^2-\frac{1}{n}\big(\eta_t^n(x)-\eta_t^n(x+\tfrac{1}{n})^2\big) \mc Y_t^n(f_x^n)^2\big\}\Big| \leq \frac{c_3(f)^2}{n^3} + \frac{c_1(f) c_3(f)}{n}.
\end{equation}
We conclude that
\begin{equation}
\lim_{n \to \infty} \Big\{\<\mc W_t^n(f)\> - \int_0^t \frac{1}{n} \sum_{x \in \bb T_n} c_x(\eta_s^n) \mc Y_s^n(f_x^n)^2 ds \Big\} =0
\end{equation}
in the sense that this difference is uniformly bounded by a deterministic sequence which goes to $0$ as $n \to \infty$. Now we explain how to change $f_x^n(z)$ by $\partial_2 f(z,x)$ in the sum above. Notice that 
\begin{equation}
\label{ec3.34}
\begin{split}
\bb E_n \big[\big|\mc Y_t^n(f_1)^2 -\mc Y_t^n(f_2)^2\big|\big]
		&\leq \bb E_n\big[\big(\mc Y_t^n(f_1)-\mc Y_t^n(f_2)\big)^2\big]^{\frac{1}{2}} \times\\
		&\quad \times \big(2\big\{\bb E_n\big[\mc Y_t^n(f_1)^2\big]+\bb E_n\big[\mc Y_t^n(f_2)^2\big]\big)^{\frac{1}{2}}.\\
\end{split}
\end{equation} 
Using the product structure of $\nu_{\frac{1}{2}}$, the right-hand side of this inequality is equal to
\begin{equation}
\frac{1}{4} \Big(\frac{1}{n} \sum_{x \in \bb T_n} \big(f_1(x)-f_2(x)\big)^2\Big)^{\frac{1}{2}} \Big(\frac{1}{n} \sum_{x \in \bb T_n} \big(f_1(x)^2 + f_2(x)^2\big)\Big)^{\frac{1}{2}}.
\end{equation}
Therefore, the application $f \mapsto \mc Y_t^n(f)^2$ from $\mc C(\bb T)$ to $L^1(\bb P_n)$ is uniformly continuous in $n$ and $t$. Let $g_x: \bb T \to \bb R$ be defined as $g_x(z) = \partial_2 f(z,x)$. We see that $f_x^n$ converges uniformly to $g_x$, the convergence being uniform in $n$ and $x$. We conclude that
\begin{equation}
\lim_{n \to \infty} \int_0^t \frac{1}{n} \sum_{x \in \bb T_n} c_x(\eta_s^n) \big( \mc Y_s^n(f_x^n)^2-\mc Y_s^n( g_x)^2\big)ds=0
\end{equation}
in $L^1(\bb P_n)$. Now we need to get rid of the term $c_x(\eta_s^n)$, which is not a function of $\mc Y_s^n$. The idea is to take advantage of the continuity of $\mc Y_t^n(g_x)$ with respect to $x$, in order to introduce an average of $c_x$ over some finite interval. Notice that for any $i \in \bb Z$, 
\begin{equation}
\lim_{n \to \infty} \sup_{x \in \bb T_n} \sup_{z \in \bb T} \big| g_x(z) -g_x(z+ \tfrac{i}{n})\big| =0.
\end{equation}
Using this limit in conjunction with \eqref{ec3.34}, we see that for any $\ell \in \bb N$,
\begin{equation}
\lim_{n \to \infty} \int_0^t \frac{1}{n} \sum_{x \in \bb T_n} c_x(\eta_s^n) \Big(\mc Y_s^n(g_x)^2-\frac{1}{2\ell+1} \sum_{i=-\ell}^\ell \mc Y_s^n(g_x(\cdot+\tfrac{i}{n}))^2\Big)ds=0
\end{equation}
in $L^1(\bb P_n)$. Passing the sum over $i$ to the $c_x$-term, we see that 
\begin{equation}
\lim_{n \to \infty} \Big\{\<\mc W_t^n(f)\> - \int_0^t \frac{1}{n} \sum_{x \in \bb T_n}\Big( \frac{1}{2\ell+1} \sum_{i=-\ell}^\ell c_{x+i}(\eta_s^n)\Big) \mc Y_s^n(g_x)^2 ds \Big\} =0
\end{equation}
in $L^1(\bb P_n)$. The expectation of $c_x(\eta)$ is equal to $\smash{\frac{1}{2}}$. Moreover, $c_x$ and $c_y$ are independent under $\nu_{\frac{1}{2}}$ as soon as $|y-x| >\smash{\frac{1}{n}}$. If we split the sum
\begin{equation}
 \frac{1}{2\ell+1} \sum_{i=-\ell}^\ell c_{x+i}(\eta)
\end{equation}
into two sums, one running over even values of $i$ and another running over odd values of $i$, we can use the elementary inequality $|a+b|^{2p} \leq 2^{p-1}(|a|^{2p}+|b|^{2p})$ plus Burkholder's inequality on each sum to get the bound
\begin{equation}
E_{\frac{1}{2}}\Big[\Big|\frac{1}{2\ell+1} \sum_{i=-\ell}^\ell c_{x+i}(\eta)-\frac{1}{2}\Big|^{2p}\Big] \leq \frac{C_p}{\ell^p}
\end{equation}
for any $p \geq 1$ and any $\ell \in \bb N$, where $C_p$ is a constant that depends only on $p$. In particular, using H\"older's inequality (to split  the product) and Burkholder's inequality (for $\mc Y_s^n(g_x)$), we get the bound
\begin{equation}
\bb E_n\Big[ \Big|\frac{1}{2\ell+1} \sum_{i=-\ell}^\ell c_{x+i}(\eta_s^n)-\frac{1}{2}\Big| \mc Y_s^n(g_x)^2\Big] \leq \frac{c_3(f)}{\sqrt \ell}
\end{equation}
for some constant $c_3(f)$ which does not depend on $n \in \bb R$ or $s \in [0,T]$. We conclude that
\begin{equation}
\varlimsup_{\ell \to \infty} \varlimsup_{n \to \infty} \int_0^t \frac{1}{n} \sum_{x \in \bb T_n} \frac{1}{2\ell+1} \sum_{i=-\ell}^\ell \big(c_{x+i}(\eta_s^n)-\tfrac{1}{2}\big) \mc Y_s^n(g_x)^2 ds =0
\end{equation}
in $L^1(\bb P_n)$. In particular, 
\begin{equation}
\lim_{n \to \infty} \Big\{ \<\mc W_t^n(f)\> - \frac{1}{2} \int_0^t \frac{1}{n} \sum_{x \in \bb T_n} \mc Y_s^n(g_x)^2 ds\Big\} =0
\end{equation}
in $L^1(\bb P_n)$. Both terms in this limit are convergent in distribution when $n \to \infty$. We summarize what we have proved up to here in the following theorem.

\begin{theorem}
\label{t3.5}
For any $f \in \mc C^\infty$ and any $t \in [0,T]$, 
\begin{equation}
\label{ec3.46}
\<\mc W_t(f)\> = \frac{1}{2}\int_0^t \int\limits_{\bb T} \mc Y_s(g_x)^2 dx ds,
\end{equation}
where $g_x: \bb T \to \bb R$ is defined as $g_x(z) = \partial_2 f(z,x)$.
\end{theorem}

\begin{remark} Recall that by definition $Q_t(f) = Q_t(f_s)$, where $f_s(x,y) = \smash{\frac{1}{2}}(f(x,y)+f(y,x))$ is the symmetric part of $f$. The expression for $\<\mc W_t(f)\>$ does not look very symmetric in the coordinates $(x,y)$.  Notice that for any smooth, symmetric function $f: \bb T^2 \to \bb R$,
\begin{equation}
\partial_2 f(x,y) = \partial_1 f(y,x).
\end{equation}
This relation allows us to write $\<\mc W_t(f)\>$ in the more symmetric way
\begin{equation}
\<\mc W_t(f)\> = \frac{1}{4} \int_0^t \int\limits_{\bb T} \big(\mc Y_s(\partial_1 f(x,\cdot))^2+\mc Y_s(\partial_2 f(\cdot,x))^2\big)dx ds.
\end{equation}
\end{remark}

\subsection{The martingale problem}
\label{s3.3}
Putting Theorems \ref{t3.3} (actually Corollaries \ref{c3.4} and \ref{c3.5}) and \ref{t3.5} together, we see that the process $\{Q_t; t \in [0,T]\}$ satisfies the following martingale problem.
\begin{itemize}
\item[\textbf{ (MP)}] For any function $f \in \mc C^\infty(\bb T^2)$, the process
\begin{equation}
\mc W_t(f) = Q_t(f) - Q_0(f) - \int_0^t Q_s(\Delta f) ds
\end{equation}
is a continuous martingale of quadratic variation
\begin{equation}
\<\mc W_t(f)\> = \frac{1}{4}\int_0^t \int\limits_{\bb T} \big(\mc Y_s(\partial_1f(x,\cdot))^2+\mc Y_s(\partial_2 f(\cdot,x))^2\big) dx ds.
\end{equation}
\end{itemize}
Very formally speaking, this is the martingale problem associated to the SPDE
\begin{equation}
d Q_t(x,y) = \Delta Q_t(x,y) dt + \tfrac{1}{2\sqrt 2} \big( \mc Y_t(x)  \nabla\tilde{\omega}_t(y) + \mc Y_t(y) \nabla \tilde{\omega}_t(x)\big) dt,
\end{equation}
where $\{\tilde{\omega}_t; t \in [0,T]\}$ is a space-time white noise. The problem comes from the fact that this martingale problem does not tell us anything about the relation between the processes $\{\mc Y_t; t \in [0,T]\}$ and $\{\tilde{\omega}_t; t \in [0, T]\}$. We would like to say that $\{\tilde{\omega}_t; t \in [0,T]\}$ is the driving process of $\{\mc Y_t; t \in [0,T]\}$, that is, the white noise appearing in \eqref{ec1.7}. Let us consider a test function of the form $f(x,y) = f_1(x) f_2(y)$ for some functions $f_1,f_2 \in \mc C^\infty(\bb T)$. Then, by \eqref{ec3.46} we have
\begin{equation}
\<\mc W_t(f)\> = \tfrac{1}{2}\|\nabla f_2\|^2\int_0^t \mc Y_s(f_1)^2 ds.
\end{equation}
In the other hand, by definition of $\{Q_t^n;t \in [0,T]\}$, we see that
\begin{equation}
Q_t^n(f) = \mc Y_t^n(f_1) \mc Y_t^n(f_2) -\frac{1}{4n} \sum_{x \in \bb T_n} f_1(x) f_2(x).
\end{equation}
Taking the limit along the subsequence $n$, we conclude that
\begin{equation}
Q_t(f) = \mc Y_t(f_1) \mc Y_t(f_2) - \int\limits_{\bb T} f_1(x)f_2(x) dx.
\end{equation}
Using It\^o's formula, we can obtain the martingale decomposition of the process $\{\mc Y_t(f_1)\mc Y_t(f_2); t \in [0,T]\}$. In particular,
\begin{equation}
\begin{split}
\mc Y_t(f_1)\mc Y_t(f_2) 
		&= \mc Y_0(f_1)\mc Y_0(f_2) + \int_0^t \big\{\mc Y_s(\Delta f_1)\mc Y_s(f_2)+\mc Y_s( f_1)\mc Y_s(\Delta f_2)\big\} ds\\
		&\quad + \tfrac{1}{\sqrt 2} \int_0^t\big\{ \mc Y_s(f_1) d\omega_s(\nabla f_2) +\mc Y_s(f_2) d \omega_s(f_1)\big\} \\
		&\quad+ \frac{t}{2} \int\limits_{\bb T} \nabla f_1(x) \nabla f_2(x) dx.
\end{split}
\end{equation}
Let us rewrite this identity in terms of the process $\{Q_t; t \in [0,T]\}$. First notice that $\Delta f = f_1 \Delta f_2 + f_2 \Delta f_1$. Then notice that the term $\frac{t}{2} \int \nabla f_1 \nabla f_2$ is exactly the normalization term missing in the integral involving the Laplacians. We conclude that
\begin{equation}
\label{ec3.56}
Q_t(f) = Q_0(f) + \int_0^t Q_s(\Delta f) ds +\tfrac{1}{\sqrt 2} \int_0^t\big\{ \mc Y_s(f_1) d\omega_s(\nabla f_2) +\mc Y_s(f_2) d \omega_s(\nabla f_1)\big\} 
\end{equation}
for any function of the form $f(x,y) = f_1(x) f_2(y)$. Therefore, at least for functions of this product form, we are able to identify $\tilde \omega$ with $\omega$. Actually, since the set of linear combinations of functions of the form $f_1(x) f_2(y)$ is dense in $\mc C^\infty(\bb T^2)$, this relation allows us to identify the martingale appearing in \textbf{(MP)} in terms of $\omega$. Let us explain this in a more rigorous way.  Let us define the $\mc D(\bb T)$-valued process $\{\mc N_t; t \in [0,T]\}$ as
\begin{equation}
\mc N_t(f) = \mc Y_t(f) - \mc Y_0(f) - \int_0^t \mc Y_s(\Delta f) ds
\end{equation}
for any $f \in \mc C^\infty(\bb T)$. By the definition of $\{\mc Y_t; t \in [0,T]\}$, the process $\{\mc N_t(f); t \in [0,T]\}$ is a continuous martingale of quadratic variation $\<\mc N_t(f)\> = \frac{1}{2} t \int (\nabla f)^2$. The relation \eqref{ec3.56} can be rewritten as
\begin{equation}
\label{ec3.58}
\mc W_t(f) = \tfrac{1}{\sqrt 2} \int_0^t \big\{ \mc Y_s(f_1) d \mc N_s(f_2)+\mc Y_s(f_2) d \mc N_s(f_1)\big\}.
\end{equation}
Notice that a by-product of our proof of tightness for $\{\mc Y_t^n; t \in [0,T]\}_{n \in \bb N}$ is a proof of tightness for $\{\mc W_t^n; t \in [0,T]\}$ as a $\mc D(\bb T^2)$-valued martingale. Therefore, $\{\mc W_t; t \in [0,T]\}$ is a well-defined, $\mc D(\bb T^2)$-valued martingale process. In particular, its distribution is determined by the values of $\{\mc W_t(f); t \in [0,T]\}$ for $f$ of the form $f_1(x)f_2(y)$. In \eqref{ec3.58} there is no mention to the process $\{Q_t; t \in [0,T]\}$. We are finally ready to establish a uniqueness result for the martingale problem \textbf{(MP)}. We state it as a theorem.

\begin{theorem}
\label{t3.8}
Let $\{(Q_t,\mc Y_t,\mc N_t); t \in [0,T]\}$ be a  $\mc D(\bb T^2) \otimes \mc D(\bb T) \otimes \mc D(\bb T)$-valued, continuous processes. Assume that
\begin{itemize}
\item[a)] for any $f \in \mc C^\infty(\bb T)$, the process $\{\mc N_t(f); t \in [0,T]\}$ is a continuous martingale of quadratic variation $\frac{1}{2} t \int (\nabla f)^2$,

\item[b)] for any $f \in \mc C^\infty(\bb T)$, the process $\{\mc Y_t(f); t \in [0,T]\}$ satisfies the relation
\begin{equation}
\mc Y_t(f) = \mc Y_0(f) + \int _0^t \mc Y_s(\Delta f) ds + \mc N_t(f),
\end{equation}
\item[c)] there exists a $\mc D(\bb T^2)$-valued process $\{\mc W_t; t \in [0,T]\}$ such that for any $f, g \in \mc C^\infty(\bb T)$,
\begin{equation}
\mc W_t(f(x)g(y)) = \tfrac{1}{\sqrt 2} \int_0^t \big\{ \mc Y_s(f) d\mc N_s(g)+ \mc Y_s(g) d \mc N_s(f)\big\},
\end{equation}
\item[d)] for any $f \in \mc C^\infty(\bb T^2)$, we have
\begin{equation}
\label{ec3.61}
Q_t(f) = Q_0(f) + \int_0^t Q_s(\Delta f) ds + \mc W_t(f),
\end{equation}
\item[e)] for any $f \in \mc C^\infty(\bb T^2)$ and any $t \in [0,T]$, the real-valued random variable $Q_t(f)$ has a Gaussian distribution of mean zero and variance $\frac{1}{4} \iint f(x,y)^2 dx dy$.
\end{itemize}
Then the distribution of $\{(Q_t,\mc Y_t,\mc N_t); t \in [0,T]\}$ is uniquely determined.
\end{theorem}

\begin{proof}
The proof was basically accomplished above. Notice that the tightness result shows existence of the triple $\{(Q_t,\mc Y_t, \mc N_t); t \in [0,T]\}$. 
Condition e) implies {\em stationarity} of the process $\{Q_t; t \in [0,T]\}$, and could be relaxed to some moment bound, but since we are only interested on the stationary case, we will not discuss these generalizations here. Using this stationarity condition, it is not hard to show that for any $t \in [0,T]$ and any smooth path $F:[0,t] \to \mc C^\infty(\bb T^2)$,
\begin{equation}
Q_t(F_t) = Q_0(F_0) + \int_0^t Q_s\big((\partial_s+\Delta)F_s\big)ds + \int_0^t d \mc W_s(F_s).
\end{equation}
This relation allows us to use Duhamel's formula to compute $Q_t(f)$ in terms of the process $\{\mc W_t; t \in [0,T]\}$. Fix $f \in \mc C^\infty(\bb T^2)$ and $t \in [0,T]$, and consider the test function $F_s = P_{t-s} f$, where $\{P_t; t \in [0,T]\}$ is the semigroup generated by $\Delta$. Then, $(\partial_s + \Delta) F_s=0$ and
\begin{equation}
\label{ec3.63}
Q_t(f) = Q_0(P_t f) + \int_0^t d \mc W_s(P_{t-s}f).
\end{equation}
Therefore, given the process $\{\mc W_t; t \in [0,T]\}$ and the initial distribution of $Q_0$, the process $\{Q_t; t \in [0,T]\}$ is uniquely determined by this Duhamel's formula, which shows uniqueness in distribution.
\end{proof}

We finish this section with the proof of Theorem \ref{t1}.

\begin{proof}[Proof of Theorem \ref{t1}] In Section \ref{s2} we showed tightness of the sequence of processes $\{Q_t^n; t \in [0,T]\}_{n \in \bb N}$. Then in this section we first showed that any limit point $\{Q_t; t \in [0,T]\}$ of this sequence satisfies the martingale problem stated in Theorem \ref{t3.8}, and then we proved that this martingale problem has a unique solution in distribution. Therefore, we conclude that the limit point of $\{Q_t^n; t \in [0,T]\}$ is unique, thus proving actual sequential convergence to that unique limit point of the whole sequence.
\end{proof}

\section{Properties of the quadratic fluctuation field at the diagonal}
\label{s4}
In this section we prove Theorems \ref{t2}, \ref{t4.5} and \ref{t2.6}. The idea is to use the martingale characterization of the process $\{Q_t; t \in [0,T]\}$ obtained in Theorem \ref{t3.8}, and more precisely the martingale decomposition \eqref{ec3.61} in order to rewrite the integral process
\begin{equation}
\int_0^t Q_s\big( (f' \otimes \delta) \ast \iota_\varepsilon\big) ds
\end{equation}
as a martingale plus a function of $Q_t$ and $Q_0$. For a given function $g \in \mc C^\infty(\bb T^2)$, let $\psi_g: \bb T^2 \to \bb R$ denote the solution of the Poisson equation
\begin{equation}
\Delta \psi = g.
\end{equation}
If $\iint g(x,y) dx dy =0$, then the solution $\psi_g$ of this equation belongs to $\mc C^\infty(\bb T^2)$, and therefore we can use it as a test function. Then, by \eqref{ec3.61} we have that
\begin{equation}
\label{ec4.3}
\int_0^t Q_s(g) ds = Q_t(\psi_g) - Q_0(\psi_g) - \mc W_t(\psi_g).
\end{equation}
Now we can explain better what is the strategy of proof of Theorem \ref{t2}. For each fixed time $t$, $Q_t$ is a white noise of variance $\smash{\frac{1}{4}}$. Therefore, $E[Q_t(\psi_g)^2] = \smash{\frac{1}{4}} \|\psi_g\|^2$ for any $t \in [0,T]$. The variance of the martingale $\mc W_t$ is equal to $\smash{\frac{1}{2}}t \|\nabla \psi_g\|^2$. Therefore, \eqref{ec4.3} is stable under approximations in $\mc H_{1,2}(\bb T^2)$. We will use the following lemma.

\begin{lemma}
\label{l4.1}
Let $\psi_f^\varepsilon$ be the solution of the Poisson equation
\begin{equation}
\label{ec4.4}
\Delta \psi = (f' \otimes \delta) \ast \iota_\varepsilon
\end{equation}
satisfying $\int_{\bb T^2} \psi(x) dx =0$.
Then, there exists a constant $C$ such that 
\begin{equation}
\|\nabla(\psi_f^\varepsilon -\psi_f^\delta)\|^2 \leq C \varepsilon \|f'\|^2
\end{equation}
for any $0<\delta<\varepsilon<1$ and any $f \in \mc C^\infty(\bb T)$.
\end{lemma}

\begin{proof}
We will obtain an explicit expression for the Fourier series of the solution of \eqref{ec4.4}. For $k,m \in \bb Z$, let $\psi_{k,m}: \bb T^2 \to \bb C$ be the trigonometric polynomial given by $\psi_{k,m}(x,y) = e^{2\pi i (kx+my)}$. Then $\{\psi_{k,m}; k,m \in \bb Z\}$ is an orthonormal basis of $L^2(\bb T^2)$.
For any function $g\in L^2(\bb T^2)$, we denote by $\widehat{g}: \bb Z^2 \to \bb C$ the Fourier series of $g$: $\widehat{g}(k,m) = \int g(x,y) \bar{\psi}_{k,m}(x,y) dx dy$ for any $k,m \in \bb Z$. 
Then, $\widehat{\Delta g}(k,m) = -4 \pi^2(k^2+m^2) \widehat{g}(k,m)$. Notice that the Fourier series of $(f' \!\otimes \delta)\ast \iota_\varepsilon$ is equal to $2\pi i(k+m)\widehat{f}(k+m) \widehat{\iota}_{\varepsilon}(k,m)$. We conclude that 
\begin{equation}
\widehat{\psi}_f^\varepsilon(k,m) = \frac{-i(k+m)\widehat{f}(k+m) \widehat{\iota}_\varepsilon(k,m)}{2\pi(k^2+m^2)}.
\end{equation}
The Fourier series $\widehat \iota_\varepsilon(k,m)$ can be computed explicitly. In fact,
\begin{equation}
\widehat \iota_\varepsilon(k,m) = \frac{\sin(2\pi k \varepsilon) \sin(2\pi m \varepsilon)}{4 \pi^2 \varepsilon^2 km},
\end{equation}
where we use the convention $\sfrac{\sin x}{x}=1$ for $x=0$.
By Parseval's identity, for any function $\psi \in L^2(\bb T^2)$,
\begin{equation}
\|\psi\|^2 = \sum_{k,m \in \bb Z}\! \big|\widehat{\psi}(k,m)\big|^2.
\end{equation}
We claim that there exists a universal constant $C>0$ such that 
\begin{equation}
|\widehat \iota_\varepsilon(k,m)-\widehat \iota_\delta(k,m)| \leq C (\varepsilon-\delta)(k^2+m^2)^{1/2}
\label{ec4.9}
\end{equation}
for any $0<\delta \leq \varepsilon$. In fact, it is enough to notice that the gradient of the function $(x,y) \mapsto \frac{\sin x \sin y}{xy}$ is bounded in $\bb R^2$. Since the function $x \mapsto \frac{\sin x}{x}$ is bounded between $-1$ and $1$, we also have the trivial bound
\begin{equation}
\label{ec4.10}
\big| \widehat \iota_\varepsilon(k,m) - \widehat \iota_\delta(k,m)\big| \leq 2,
\end{equation}
which is actually better than \eqref{ec4.9} if $(k^2+m^2)^{1/2} \geq\frac{2}{C(\varepsilon-\delta)}$. Notice that
\begin{equation}
\label{ec4.11}
\|\nabla(\psi_f^\varepsilon-\psi_f^\delta)\|^2 = 
    \sum_{k,m \in \bb Z} \frac{(k+m)^2 |\widehat f(k+m) |^2}{k^2+m^2}\big| \widehat \iota_\varepsilon(k,m) - \widehat \iota_\delta(k,m)\big|^2.
\end{equation}
We will split this sum into a sum over two sets, depending whether \eqref{ec4.9} or \eqref{ec4.10} is better. We start with the case on which the trivial bound \eqref{ec4.10} is better. Notice that the set $R_1 = \{(k,m) \in \bb Z^2; (k^2+m^2)^{1/2} \geq \frac{2}{C(\varepsilon -\delta)}$ is contained on the set 
\begin{equation}
R_1'=\Big\{(k,m) \in \bb Z^2; |k+m| \geq \frac{2}{C(\varepsilon -\delta)}, |k-m| \geq \frac{2}{C(\varepsilon -\delta)}\Big\}.
\end{equation}
Let us define the new coordinates $\ell = k+m$, $n = k-m$. In these new coordinates, the set $R_1'$ is given by\footnote{As a slight abuse of notation, we use the same letter $R_1'$ for the image of the set $R_1'$ under this change of coordinates.}
\begin{equation}
R_1'=\Big\{|\ell| \geq \frac{2}{C(\varepsilon -\delta)}, |n| \geq \frac{2}{C(\varepsilon -\delta)}\Big\}.
\end{equation}
Using the estimate $k^2+m^2 \geq \frac{1}{2}n^2$, we have that
\begin{equation}
\sum_{(k,m) \in R_1'} \frac{(k+m)^2 |\widehat f(k+m) |^2}{k^2+m^2}\big| \widehat \iota_\varepsilon(k,m) - \widehat \iota_\delta(k,m)\big|^2
  \leq 8 \sum_{(\ell,n) \in R_1'} \frac{\ell^2}{n^2} |\widehat f(\ell)|^2.
\end{equation}
We conclude that there is a constant $C_1$ such that for any $\varepsilon$ small enough,
\begin{equation}
\sum_{(k,m) \in R_1'}  \frac{(k+m)^2 |\widehat f(k+m) |^2}{k^2+m^2}\big| \widehat \iota_\varepsilon(k,m) - \widehat \iota_\delta(k,m)\big|^2
      \leq C_1(\varepsilon - \delta) \| f'\|^2.
\end{equation}

Now we estimate the sum in \eqref{ec4.11} over the set  $R_2 = (R_1')^c$. Using \eqref{ec4.9}, we see that
\begin{equation}
\begin{split}
\sum_{(k,m) \in R_2}  \frac{(k+m)^2 |\widehat f(k+m) |^2}{k^2+m^2}\big| \widehat \iota_\varepsilon(k,m) - \widehat \iota_\delta(k,m)\big|^2
      & \leq C^2(\varepsilon -\delta)^2 \sum_{(\ell,n) \in R_2} \ell^2 |\widehat f(\ell)|^2 \\	
      & \leq 2 C^3 (\varepsilon-\delta) \|f'\|^2.
\end{split}
\end{equation}
Putting the two pieces together, we end the proof of the lemma.
\end{proof}

\begin{remark}
At first glance, it seems that solving the Poisson equation \eqref{ec4.4} with $f'$ instead of $f$ is not needed in order to prove this lemma, since the final bound depends only on $f'$ and not on $f$. Actually, what we need is the integral condition $\int_{\bb T} f'(x) dx =0$, which is equivalent to $f'$ being the actual derivative of a function $f$. This condition means that the $0$-th order of the Fourier's expansion of $f'$ vanishes, allowing to divide by $k^2+m^2$ without worrying about the case $k=0$, $m=0$ in various steps of the computations.
\end{remark}

What Lemma \ref{l4.1} is saying, is that for any $t \in [0,T]$ and any $f \in \mc C^\infty(\bb T)$, the sequence $\{\mc A_t^\varepsilon(f)\}_\varepsilon$ is Cauchy in $L^2(P)$, and therefore the limit
\begin{equation}
\mc A_t(f) = \lim_{\varepsilon \to 0} \mc A_t^\varepsilon(f)
\end{equation}
exists in $L^2(P)$. This does not show the existence of the $\mc D(\bb T^2)$-valued process $\{\mc A_t; t \in [0,T]\}$, nor the convergence of $\{\mc A_t^\varepsilon; t \in [0,t]\}$ to it. In order to obtain this convergence, we need a more refined argument. Notice that
\begin{equation}
\label{ec4.18}
\mc A_t^\varepsilon(f) = Q_t(\psi^\varepsilon_f)-Q_0(\psi^\varepsilon_f) - \mc W_t(\psi^\varepsilon_f).
\end{equation}
Therefore, for $0<\delta <\varepsilon$ we have 
\begin{equation}
\label{ec4.19}
\begin{split}
E\big[\big(\mc A_t^\varepsilon(f) - \mc A_t^\delta(f)\big)^2\big] 
    &\leq 2 E\big[ \big(Q_t(\psi^\varepsilon_f-\psi^\delta_f)-
Q_0(\psi^\varepsilon_f-\psi^\delta_f)\big)^2\big]\\
    &\quad +2 E[ \mc W_t(\psi^\varepsilon_f-\psi^\delta_f)^2].
\end{split}
\end{equation}
The second expectation on the right-hand side is equal to $\sfrac{1}{2} t \|\nabla(\psi^\varepsilon_f-\psi^\delta_f)\|^2$, which is bounded by $Ct\varepsilon \|f'\|^2$. The first expectation is computed in the following lemma:

\begin{lemma}
\label{l4.2}
For any $g \in L^2(\bb T^2)$, 
\begin{equation}
E\big[\big( Q_t(g)-Q_0(g)\big)^2\big] = \tfrac{1}{2} \<g,g-P_tg\>.
\end{equation}
\end{lemma}

\begin{proof}
Let us assume that $g \in \mc C^\infty(\bb T^2)$. By Duhamel's formula \eqref{ec3.63}, we have that
\begin{equation}
\begin{split}
E\big[\big( Q_t(g)-Q_0(g)\big)^2\big]
    &= E[Q_0(P_t g-g)^2] + E\big[ \big( \int_0^t d \mc W_s(P_{t-s}g)\big)^2\big]\\
    &= \tfrac{1}{4} \|P_tg-g\|^2 + \tfrac{1}{2} \int_0^t \|\nabla(P_s g)\|^2.
\end{split}
\end{equation}
Since $\|\nabla(P_s g) \|^2 = \sfrac{1}{2} \sfrac{d}{dt} \|P_t g\|^2$, the lemma follows for $g \in \mc C^\infty(\bb T^2)$. Since $Q_t$ and $P_t$ are continuous under approximations in $L^2(\bb T^2)$, the lemma follows for $g \in L^2(\bb T^2)$ by approximations.
\end{proof}

Notice that
\begin{equation}
\tfrac{d}{dt} \<g,g-P_t g\> = \|\nabla P_t g\|^2 \leq \|\nabla g\|^2
\end{equation}
and in particular $ \<g,g-P-t g\> \leq t \|\nabla g\|^2$. This shows that both expectations in \eqref{ec4.19} are of the same order. Therefore, we have the bound
\begin{equation}
\label{ec4.23}
E\big[\big(\mc A_t^\varepsilon(f) - \mc A_t^\delta(f)\big)^2\big] \leq Ct \varepsilon \|f'\|^2,
\end{equation}
valid for any $f \in \mc C^\infty(\bb T)$, any $t \in [0,T]$ and any $\varepsilon>\delta >0$. Using the representation 
\begin{equation}
\mc A_t^\varepsilon(f) = \int_0^t Q_s\big((f'\otimes \delta) \ast \iota_\varepsilon\big) ds,
\end{equation}
we get the simple bound
\begin{equation}
\label{ec4.25}
E[\mc A_t^\varepsilon(f)^2] \leq \tfrac{1}{4} t^2\|(f' \otimes \delta) \ast \iota_\varepsilon\|^2 \leq \frac{C t^2}{\varepsilon} \|f'\|^2.
\end{equation}
Taking $\delta \to 0$, choosing $\varepsilon = \sqrt{t}$ and combining \eqref{ec4.23} with \eqref{ec4.25}, we obtain the bound
\begin{equation}
\label{ec4.26}
E[\mc A_t(f)^2] \leq C t^{3/2} \|f'\|^2.
\end{equation}
These estimates are all we need in order to show that $\{\mc A_t; t \in [0,T]\}$ is a well-defined, $\mc D(\bb T^2)$-valued process and to show that $\{\mc A_t^\varepsilon; t \in [0,T]\}$ converges to $\{\mc A_t; t \in [0,T]\}$ at the level of processes, see Theorem 2.1 of \cite{GJ3} for the details. 

\subsection{Small-time asymptotics of the quadratic field}

Let us denote by $\psi_f$ the solution of the Poisson equation
\begin{equation}
\Delta \psi = f' \otimes \delta
\end{equation}
\label{ec4.27}
satisfying $\int_{\bb T^2} \psi_f dx =0$. By Lemma \ref{l4.1}, $\psi \in \mc H_{1,2}(\bb T^2)$. The various approximation results shown in the previous section allow to obtain the following representation formula for $\mc A_t(f)$:

\begin{equation}
\label{ec4.28}
\mc A_t(f) = Q_0(P_t \psi_f -\psi_f) + \int_0^t d \mc W_s(P_{t-s} \psi_f -\psi_f)
\end{equation}
for any $t \in [0,T]$. Using this formula as a starting point, we can analyze the behavior of $\mc A_t(f)$ for small times $t$. The variance of the first term on the right-hand side of \eqref{ec4.28} is equal to
$\tfrac{1}{4} \|P_t \psi_f -\psi_f\|^2$, while the variance of the second term on the right-hand side of \eqref{ec4.28} is equal to
\begin{equation}
\tfrac{1}{2} \int_0^t \|P_{t-s} \psi_f -\psi_f\|^2 ds= \tfrac{1}{2} \int_0^t \|P_s\psi_f -\psi_f\|^2 ds.
\end{equation}
If the norm $\|P_t\psi_f -\psi_f\|$ decays to $0$ like a power law, then the martingale term in \eqref{ec4.28} turns out to be of lower order than the term $Q_0(P_t \psi_f -\psi_f)$. In fact, we have the following 
\begin{lemma}
\label{l4.4}
For any smooth function $f: \bb T \to \bb R$ we have that
\begin{equation}
\lim_{t \to 0} \frac{\|P_t \psi_f -\psi_f\|^2}{t^{3/2}} = c \<f,-\Delta f\>,
\end{equation}
where
\begin{equation}
\kappa = \frac{1}{8\pi^4}\int\limits_{\bb R} \Big( \frac{1-e^{-\pi^2 x^2}}{x^2}\Big)^2 dx.
\end{equation}
\end{lemma}
\begin{proof}
Notice that for any $\psi: \bb T^2 \to \bb R$ regular enough (in fact, anything at least as regular as a measure on $\bb T^2$ works), $\widehat{P_t \psi}(k,m) = e^{-4\pi^2(k^2+m^2)t} \widehat{\psi}(k,m)$. Therefore, by Parseval identity we have that
\begin{equation}
\|P_t \psi_f -\psi_f\|^2 = \sum_{k,m \in \bb Z} \frac{(k+m)^2\big|\widehat{f}(k+m)\big|^2}{4 \pi^2 (k^2+m^2)^2} \big(1-e^{-4\pi^2 (k^2+m^2)t}\big)^2.
\end{equation}
Let us perform the change of variables $u=k+m$, $v=k-m$. Notice that $k^2+m^2= \frac{1}{2}(u^2+v^2)$. We have that
\begin{equation}
\label{ec4.32}
\|P_t \psi_f -\psi_f\|^2 = \sum_{u, v} \frac{u^2\big|\widehat{f}(u)\big|^2}{ \pi^2 (u^2+v^2)^2} \big(1-e^{-\pi^2 (u^2+v^2)t}\big)^2,
\end{equation}
where the sum is over $u, v \in \bb Z$ such that $u+v$ is even. Let us define $g: \bb R^2 \to \bb R$ as
\begin{equation}
g(x,y) = \frac{1-e^{-\pi^2(x^2+y^2)}}{x^2+y^2}.
\end{equation}
Rearranging terms in a convenient way, we can rewrite \eqref{ec4.32} as
\begin{equation}
\|P_t \psi_f -\psi_f\|^2 = t^{3/2}\!\sum_{u \in \bb Z} \frac{u^2\big|\widehat{f}(u)\big|^2}{ 2\pi^2 }\Big( 2\sqrt{t}\!\!\!\!\sum_{\substack{v \in \bb Z\\ u+v \text{ even}}} \!\!\!\! g\big( u\sqrt t, v \sqrt t\big)^2\Big).
\end{equation}
Therefore, it is enough to find the limit of the quantity
\begin{equation}
\!\sum_{u \in \bb Z} \frac{u^2\big|\widehat{f}(u)\big|^2}{ 2\pi^2 }\Big( 2\sqrt{t}\!\!\!\!\sum_{\substack{v \in \bb Z\\ u+v \text{ even}}} \!\!\!\! g\big( u\sqrt t, v \sqrt t\big)^2\Big).
\end{equation}
Notice that the sum between parentheses is a Riemann sum for the integral
\begin{equation}
\label{ec4.36}
\int\limits_{\bb R} g(u\sqrt t, y)^2 dy,
\end{equation}
the factor $2$ in front of it coming from the fact that the sum is over one half of the integers. We can check that there exists a finite constant $c$ such that
\begin{equation}
\big|g(x,y)\big| \leq \frac{c}{1+y^2}, \quad \big|\partial_y g(x,y)\big| \leq \frac{c}{1+y^4}
\end{equation}
for any $x \in \bb R$. Therefore, 
\begin{equation}
\lim_{t \to 0} \;2\sqrt t \!\!\!\!\sum_{\substack{v \in \bb Z\\ u+v \text{ even}}} \!\!\!\! g\big( x, v \sqrt t\big)^2 = \int g(x,y)^2 dy,
\end{equation}
uniformly in $x$. In particular, we have the diagonal limit
\begin{equation}
\lim_{t \to 0} 2 \sqrt t\!\!\!\!\sum_{\substack{v \in \bb Z\\ u+v \text{ even}}} \!\!\!\! g\big( u\sqrt t, v \sqrt t\big)^2 = \int g(0,y)^2 dy
\end{equation}
for any $u \in \bb Z$. Using the bound $|g(x,y)| \leq \frac{c}{1+y^2}$ we can exchange the limit and the summation to obtain that
\begin{equation}
\lim_{t \to 0} \frac{\|P_t \psi_f -\psi_f\|^2}{t^{3/2}} = \frac{1}{2 \pi^2} \sum_{u \in \bb Z} u^2 \big| \widehat{f}(u)\big|^2 \int\limits_{\bb R} g(0,y)^2 dy,
\end{equation}
which proves the lemma.
\end{proof}

With Lemma \ref{l4.4} in hands, we can prove Theorem \ref{t4.5}:
\begin{proof}[Proof of Theorem \ref{t4.5}]
Recall the discussion after \eqref{ec4.28}. Since $\|P_{\varepsilon t} \psi_f -\psi_f\|$ decays to $0$ as $\varepsilon^{3/2}$, the second term on the right-hand side of \eqref{ec4.28} is of smaller order than the first term. In other words,
\begin{equation}
\lim_{\varepsilon \to 0} \varepsilon^{-3/4} \mc A_{\varepsilon t} (f) = \lim_{\varepsilon \to 0} \varepsilon^{-3/4} Q_0(P_t \psi_f -\psi_f), 
\end{equation}
whenever one of the limits exists. The process $Q_0$ is Gaussian and centered, and therefore it is easy to check its convergence. In fact, this is exactly what Lemma \ref{l4.4} accomplishes. We conclude that $\varepsilon^{-3/4} \mc A_{\varepsilon t} (f) $ converges in distribution as $\varepsilon \to 0$ to a Gaussian random variable of mean zero and variance $\frac{c}{4} \<f,-\Delta f\>t^{3/2}$. Since the process $\{\mc A_t(f); t\geq 0\}$ has stationary increments, we have proved convergence in the sense of finite-dimensional distributions. In order to get convergence at the level of processes, we only need to show tightness of the sequence $\{\varepsilon^{-3/4}\mc A_{\varepsilon t}(f); t \geq 0\}$. But this follows from the comment after \eqref{ec4.26}.
\end{proof}
Notice that the process $\{\bb A_t(f); t\geq 0\}$ is a {\em fractional Brownian motion} of Hurst exponent $\frac{3}{4}$. This is exactly the same process arising as the scaling limit of occupation times of the exclusion process \cite{GJ3}. This fact is a quantitative version of the formal claim which states that fluctuations of additive functionals of local functions of degree 1 (like the occupation time) are of the same nature as fluctuations of additive functionals of extensive functions of degree 2 (like the quadratic field). 

\subsection{Large-time variance of the quadratic field} 

In the previous section we have seen how we can extract non-trivial information about the quadratic field $\{\mc A_t; t \geq 0\}$ for small times using the representation \eqref{ec4.28}. Now we will see what can we say about the variance of $\mc A)t(f)$ when $t$ is large. Actually, it will be easier to work with the representation
\begin{equation}
\mc A_t(f) = Q_t(\psi_f) - Q_0(\psi_f) - \mc W_t(\psi_f).
\end{equation}
This relation can be obtained by passing to the limit in \eqref{ec4.18}. The variance of $Q_t(\psi_f)$ (and also of $Q_0(\psi_f)$) is equal to $|\psi_f|^2$, while the variance of $\mc W_t(\psi_f)$ is equal to $\frac{1}{2}t \|\nabla \psi_f\|^2$. Therefore, we see that
\begin{equation}
\lim_{t \to \infty} \frac{E\big[\mc A_t(f)\big]}{t} = \tfrac{1}{2} \|\nabla \psi_f \|^2.
\end{equation}
%%%
The following lemma tells us the behavior of $\| \psi_f \|^2$:

\begin{lemma}
\label{l4.6} There exists a function $a: \bb Z \to \bb R$ such that
\begin{equation}
\lim_{|k| \to \infty} a(k) =\frac{\pi}{2}
\end{equation}
and such that
\begin{equation}
\|\psi_f\|^2 = \sum_{k \in \bb Z} a(k) |k| \big|\widehat{f}(k)\big|^2.
\end{equation}
\end{lemma}
\begin{proof}
Notice that
\begin{equation}
\|\psi_f\|^2 = \sum_{k,m} \frac{(k+m)^2\big| \widehat{f}(k+m)\big|^2}{k^2+m^2}. 
\end{equation}
Using the change of variables $u=k+m$, $v =k-m$ and rearranging terms, we can rewrite this expression as
\begin{equation}
\sum_{u \in \bb Z} |u| \big|\widehat{f}(u)\big|^2 \!\!\!\!\! \sum_{\substack{v \in \bb Z\\ u+v \text{ even}}} \!\!\!\! \frac{1}{|u|} \frac{1}{1+\big(\tfrac{v}{|u|}\big)^2}.
\end{equation}
The second sum is a Riemann sum for $\frac{\pi}{2}=\int (1+x^2)^{-1} dx$. Therefore, it is enough to choose
\begin{equation}
a(u) = \frac{1}{|u|}  \!\!\!\!\! \sum_{\substack{v \in \bb Z\\ u+v \text{ even}}} \!\!\!\! \frac{1}{|u|} \frac{1}{1+\big(\tfrac{v}{|u|}\big)^2}.
\end{equation}
\end{proof}

The fact that $a(k) \to \frac{\pi}{2}$ allows us to identify the limiting variance of $\mc A_t(f)$ as the energy associated to a sort of {\em fractional Laplacian operator} of order $\frac{1}{4}$. In a more precise way, we can write 
\begin{equation}
\|\psi_f \|^2 = \<f,-\tfrac{1}{\sqrt{8\pi}}(-\Delta)^{1/4}f + \mc K f\>, 
\end{equation}
where $\mc K$ is a continuous operator in $L^2(\bb T)$. The appearance of this compact operator is just a finite-volume effect.

\section*{Acknowledgements} 

This work was written during a visit of the author to Universit\'e Aix-Marseille.
The author would like to thank the warm hospitality at the Institut de Math\'emati-ques de Marseille and to Etienne Pardoux for valuable comments and for helping with the proof of Theorem \ref{t3.8}.

\bibliographystyle{plain}

\begin{thebibliography}{10}

\bibitem{Ass}
Sigurd Assing.
\newblock A limit theorem for quadratic fluctuations in symmetric simple
  exclusion.
\newblock {\em Stochastic Process. Appl.}, 117(6):766--790, 2007.

\bibitem{Cor}
Ivan Corwin.
\newblock The {K}ardar-{P}arisi-{Z}hang equation and universality class.
\newblock {\em Random Matrices Theory Appl.}, 1(1):1130001, 76, 2012.

\bibitem{D-MIPP}
A.~{De Masi}, N.~Ianiro, A.~Pellegrinotti, and E.~Presutti.
\newblock A survey of the hydrodynamical behavior of many-particle systems.
\newblock In {\em Nonequilibrium phenomena, {II}}, Stud. Statist. Mech., XI,
  pages 123--294. North-Holland, Amsterdam, 1984.

\bibitem{FPV}
P.~A. Ferrari, E.~Presutti, and M.~E. Vares.
\newblock Nonequilibrium fluctuations for a zero range process.
\newblock {\em Ann. Inst. H. Poincar{\'e} Probab. Statist.}, 24(2):237--268,
  1988.

\bibitem{GJ}
Patr{\'i}cia Gon\c{c}alves and Milton Jara.
\newblock Nonlinear fluctuations of weakly asymmetric interacting particle
  systems.
\newblock {To appear in \em Arch. Ration. Mech. Anal.}

\bibitem{GJ3}
Patr{\'i}cia Gon\c{c}alves and Milton Jara.
\newblock Scaling limits of additive functionals of interacting particle
  systems.
\newblock {\em Comm. Pure Appl. Math.}, 66(5):649--677, 2013.

\bibitem{GJS}
Patr\'icia Gon\c{c}alves, Milton Jara, and Sunder Sethuraman.
\newblock A stochastic burgers equation from a class of microscopic
  interactions.
\newblock {To appear in \em Ann. Probab.}

\bibitem{Hai1}
Martin Hairer.
\newblock Solving the {KPZ} equation.
\newblock {\em Ann. of Math. (2)}, 178(2):559--664, 2013.

\bibitem{Hai2}
Martin Hairer.
\newblock A theory of regularity structures.
\newblock {\em Preprint}, 2013.

\bibitem{Mit}
Itaru Mitoma.
\newblock Tightness of probabilities on {$C([0,1];{\mathcal S}\sp{\prime} )$}
  and {$D([0,1];{\mathcal S}\sp{\prime} )$}.
\newblock {\em Ann. Probab.}, 11(4):989--999, 1983.

\bibitem{RhoVar}
R\'emy Rhodes and Vincent Vargas.
\newblock Gaussian multiplicative chaos and applications: a review.
\newblock {\em Preprint}, 2013.

\bibitem{Spo}
Herbert Spohn.
\newblock Stochastic integrability and the kpz equation.
\newblock {\em Preprint}, 2012.

\end{thebibliography}

% \begin{appendix}
% \section{}
% 
% In this section we prove estimate \eqref{ec4.9}. Fix $k, m \in \bb Z\setminus \{0\}$. Then, $\hat \iota_\varepsilon (k,m) = h(\epsilon$, where $h: \bb R \to \bb R$ is given by
% \begin{equation}
% h(x) = \frac{\sin( 2\pi k x) \sin(2 \pi m x)}{4 \pi^2 k m x^2}.
% \end{equation}
% Therefore,
% 
% 
% \end{appendix}

\end{document}